\documentclass{ims9x6}

\usepackage{amsfonts}
\def\calF{{\cal F}}
\def\calS{{\cal S}}
\def\calD{{\cal D}}

\def\calB{{\cal B}}
\def\calA{{\cal A}}
\def\calC{{\cal C}}
\def\calL{{\cal L}}
\def\calP{{\cal P}}
\def\calH{{\cal H}}

\def\calG{{\cal G}}

\def\calN{{\cal N}}

\def\n{{\mathbb N}}
\def\r{{\mathbb R}}

\def\c{{\mathbb C}}
\def\e{{\mathbb E}}

\usepackage{mathrsfs}

\def\scrB{{\mathscr B}}
\def\scrC{{\mathscr C}}

\def\scrF{{\mathscr F}}
\def\scrG{{\mathscr G}}
\def\scrH{{\mathscr H}}

\usepackage{amsmath}
\usepackage{amssymb, color}

\def\ll{\langle\!\langle}
\def\gg{\rangle\!\rangle}
\def\hat{\widehat}
\def\tilde{\widetilde}

\def\l{\langle}
\def\g{\rangle}

\def\I{{\rm i}}

\makeindex



\begin{document}

\setcounter{page}{61}

\chapter*{NORMAL APPROXIMATION FOR WHITE NOISE FUNCTIONALS BY STEIN'S METHOD AND\\ HIDA CALCULUS}

\markboth{L.~H.~Y.~Chen, Y.-J.~Lee \& H.-H.~Shih}{Stein's Method and Hida Calculus}

\author{Louis H.~Y.~Chen}
\address{Department of Mathematics\\
National University of Singapore\\
10 Lower Kent Ridge Road, Singapore 119076 \\
matchyl@nus.edu.sg}

\author{Yuh-Jia Lee}
\address{Department of Applied Mathematics\\
National University of Kaohsiung\\
Kaohsiung, Taiwan 811\\
yuhjialee@gmail.com}

\author{Hsin-Hung Shih}
\address{Department of Applied Mathematics\\
National University of Kaohsiung\\
Kaohsiung, Taiwan 811\\
hhshih@nuk.edu.tw}

\begin{abstract}
In this chapter, we establish a framework for normal approximation for white noise functionals by Stein's method and Hida calculus.
Our work is inspired by that of Nourdin and Peccati (\textit{Probab. Theory Relat. Fields}
145 (2009), 75-118), who combined Stein's method and Malliavin calculus for normal approximation for functionals of Gaussian processes.
\end{abstract}

\section{Introduction}

Stein's method, introduced by C.~Stein \cite{Steina} in his 1972 paper, is a powerful way of determining the accuracy of normal approximation to the distribution of a sum of dependent random variables.
It has been extended to approximations by a broad class of other probability distributions such as the Poisson, compound Poisson, and Gamma distributions, and to approximations on finite as well as infinite dimensional spaces.
The results of these approximations have been extensively applied in a wide range of other fields such as the theory of random graphs, computational molecular biology, etc.
For further details, see \cite{{BA},{BB},{BC},{Chen},{Sch},{Schb},{Steinb}} and the references cited therein.

Analysis on infinite-dimensional Gaussian spaces has been formulated in terms of Malliavin calculus and Hida calculus.
The former, introduced by Malliavin (\cite{Ma}), studies the calculus of Brownian functionals and their applications on the classical Wiener space.
The connection between Stein's method and Malliavin calculus has been explored by Nourdin and Peccati (see \cite{NP}).
They developed a theory of normal approximation on infinite-dimensional Gaussian spaces.
In their connection, the Malliavin derivative $D$ plays an important role.
See also \cite{{Chena},{CP}}.

Hida calculus, also known as white noise analysis, is the mathematical theory of white noise initiated by T.~Hida in his 1975 Carleton Mathematical Lecture Notes \cite{Hida}.
Let $\{B(t);\,t\in\r\}$ be a standard Brownian motion and let the white noise $\dot{B}(t)\equiv dB(t)/dt, t\in \r$, be represented by generalized functions.
By regarding the collection $\{\dot{B}(t);\,t\in\r\}$ as a coordinate system, Hida defined and studied generalized white noise functionals $\varphi(\dot{B}(t),\,t\in\r)$ through their $U$-functionals.
We refer the interested reader to \cite{{HKPS},{HidaSi},{Kb},{Si}}.

The objective of this chapter is to develop a connection between Stein's method and Hida calculus  for normal approximation for white noise functionals (see Section 5).
Our approach is analogous to that for the connection between Stein's method and Malliavin calculus as established by Nourdin and Peccati \cite{NP}.
The connection between Stein's method and Hida calculus will be built on the expression of the number operator (or the Ornstein-Ulenbeck operator) in terms of the Hida derivatives through integration by parts techniques (see Section 4).
The difficulty that we have encountered so far is that the Hida derivative $\partial_t$, that is the $\dot{B}(t)$-differentiation, cannot be defined on all square-integrable white noise functionals in $(L^2)$.
Extending the domain of $\partial_t$ to a larger subclass of $(L^2)$ and studying the regularity of  $\partial_t$ will be a key contribution in our chapter.

At the time of completing this chapter we came to know about the PhD thesis of Chu \cite{Chu}, in which he developed normal approximation (in Wasserstein distance) for L\'evy functionals by applying Stein's method and Hida calculus.
He achieved this by using the white noise approach of Lee and Shih \cite{LS2004}.

We list some notations which will be often used in this chapter.

\vskip-\lastskip
\pagebreak

\begin{notation}
\begin{itemize}
\item[$(1)$]
For a real locally convex space $V$, $\scrC V$ denotes its complexification.
If $(V,\,|\cdot|_V)$ is a real Hilbert space, then $\scrC V$ is a complex Hilbert space with the $|\cdot|_{\scrC V}$-norm given by $|\phi|_{\scrC V}^2=|\phi_1|_V^2+|\phi_2|_V^2$ for any $\phi=\phi_1+\I\,\phi_2$, $\phi_1,\phi_2\in V$.
Specially, for $V=\calS_p$ with $|\cdot|_p$-norm (see Section 3 for the definition), we will still use $|\cdot|_p$ to denote $|\cdot|_{\scrC\calS_p}$.

\item[$(2)$]
The symbol $(\cdot,\cdot)$ denotes the $\calS'$-$\calS$, $\scrC\calS'$-$\scrC\calS$, $\calS_{-p}$-$\calS_p$, or $\scrC\calS_{-p}$-$\scrC\calS_p$ pairing.

\item[$(3)$]
For a $n$-linear operator $T$ on $X\times\cdots\times X$, $Tx^n$ means $T(x,\ldots,x)$ as well as $Tx^{n-1}y=T(x,\ldots,x,y)$, $x,y\in X$, where $X$ is a real or complex locally convex space.

\item[$(4)$]
The constant $\omega_r$ with $r>0$ is given by the square root of $\sup\{n/2^{2nr};\,n\in\n_0\}$, where $\n_0=\n\cup\{0\}$.
\end{itemize}
\end{notation}

\section{Stein's Method}

In this section we give a brief exposition of the basics of Stein's method for normal approximation.

\subsection{From characterization to approximation}

In his 1986 monograph \cite{Steina}, Stein proved the following characterization of the normal distribution.

\begin{proposition}(Stein's lemma)\label{prop-characterization}
The following are equivalent.\\
(i) $W \sim \mathcal{N}(0,1)$;\\
(ii) $\e[f'(W) - Wf(W)] = 0$  for all $f \in \mathcal{C}_B^1$.
\end{proposition}

\begin{proof}
By integration by parts, (i) implies (ii).
If (ii) holds, solve
\begin{equation*}
f'(w) - wf(w) = h(w) - \e h(Z)
\end{equation*}
where $h\in \mathcal{C}_B$ and $Z$ has the standard normal distribution (denoted by $Z \sim N(0,1)$).
Its bounded unique solution $f_h$ is given by
\begin{eqnarray} \label{Stein-eqn-soln} \notag
f_h(w) &=& - e^{\frac{1}{2}w^2}\int_w^\infty e^{-\frac{1}{2}t^2}[h(t)-\e h(Z)]dt \\
&=& e^{\frac{1}{2}w^2}\int_{-\infty}^w e^{-\frac{1}{2}t^2}[h(t)-\e h(Z)]dt.
\end{eqnarray}
Using $\int_w^\infty e^{-\frac{1}{2}t^2}dt \le w^{-1}e^{-\frac{1}{2}w^2}$ for $w > 0$, we can show that $f_h \in \mathcal{C}_B^1$ with $\|f_h\|_\infty \le \sqrt{2\pi e}\|h\|_\infty$ and $\|f'_h\|_\infty \le 4\|h\|_\infty$.
Substituting $f_h$ for $f$ in (ii) leads to
$$
\e h(W) = \e h(Z) \quad \text{for} \quad h \in \mathcal{C}_B.
$$
This proves (i).
\end{proof}

The equation
\begin{equation} \label{char-normal}
\e Wf(W) = \e f'(W)
\end{equation}
for all $f \in \calC_B^1$, which characterizes the standard normal distribution, is called the Stein identity for normal distribution.
In fact, if $W \sim N(0,1)$, (\ref{char-normal}) holds for absolutely continuous $f$ such that $\e|f'(W)| < \infty$.

Let $W$ be a random variable with $\e W = 0$ and $\mathrm{Var}(W) = 1$.
Proposition \ref{prop-characterization} suggests that the distribution of $W$ is ``close'' to $N(0,1)$ if and only if
$$
\e[f'(W) - Wf(W)] \simeq 0
$$
for $f\in \mathcal{C}_B^1$.
How ``close'' the distribution of $W$ is to the standard normal distribution may be quantified by determining how close $\e[ f'(W)-Wf(W)]$ is to $0$.
To this end we define a distance between the distribution of $W$ and the standard normal distribution as follows.
\begin{equation} \label{def-d_G}
d_\scrG(W,Z) := \sup_{h\in\scrG}|\e h(W)-\e h(Z)|
\end{equation}
where $\scrG$ is a separating class and the distance $d_\scrG$ is said to be induced by $\scrG$.  By a separating class $\scrG$, we mean a class of Borel measurable real-valued functions defined on $\r$ such that two random variables, $X$ and $Y$, have the same distribution if $\e h(X) = \e h(Y)$ for $h \in \scrG$.
Such a separating class contains functions $h$ for which both $\e h(X)$ and $\e h(Y)$ exist.

Let $f_h$ be the solution, given by (\ref{Stein-eqn-soln}), of the Stein's equation
\begin{equation} \label{Stein-equation}
f'(w) - wf(w) = h(w) - \e h(Z)
\end{equation}
where $h\in \scrG$.
Then we have
\begin{equation}\label{error-term}
d_\scrG(W,Z) = \sup_{h\in\scrG}|\e [f'_h(W) - Wf_h(W)]|.
\end{equation}
So bounding the distance $d_\scrG(W,Z)$ is equivalent to bounding $\sup_{h\in\scrG}|\e[f'_h(W) - Wf_h(W)]|$, for which we need to study the boundedness properties of $f_h$ and the probabilistic structure of $W$.

The following three separating classes of Borel measurable real-valued functions defined on $\r$ are of interest in normal approximation.
\begin{eqnarray*}
\scrG_W &:=& \{h; |h(u)-h(v)| \le |u-v|\}, \\
\scrG_K &:=& \{h; h(w) =1 ~\text{for} ~w\le x ~ \text{and} = 0  ~\text{for} ~ w> x, x \in \r\}, \\
\scrG_{TV} &:=& \{h; h(w) = I(w\in A), A ~ \text{is a Borel subset of}~ \r\}.
\end{eqnarray*}
The distances induced by these three separating classes are respectively called the Wasserstein distance, the Kolmogorov distance, and the total variation distance.
It is customary to denote $d_{\scrG_W}$, $d_{\scrG_K}$ and $d_{\scrG_{TV}}$ respectively by $d_W$, $d_K$ and $d_{TV}$.

Since for each $h$ such that $|h(x)-h(y)|\le|x-y|$, there exists a sequence of $h_n \in \mathcal{C}^1$ with $\|h_n'\|_\infty \le 1$ such that $\|h_n - h\|_\infty \rightarrow 0$ as $n \rightarrow \infty$, we have
\begin{equation} \label{d_W-h-C^1}
d_W(W,Z) = \sup_{h \in \mathcal{C}^1, \|h'\|_\infty \le 1}|\e h(W) -\e h(Z)|.
\end{equation}
By an application of Lusin's theorem, we also have
\begin{equation} \label{d_TV-h-C}
d_{TV}(W,Z) = \sup_{h\in \mathcal{C}, 0\le h \le 1}|\e h(W)-\e h(Z)|.
\end{equation}
It is also known that
\begin{eqnarray*}
d_{TV}(W,Z) &=& \frac{1}{2}\sup_{\|h\|_\infty \le 1}|\e h(W) - \e h(Z)| \\
&=& \frac{1}{2}\sup_{h \in \mathcal{C}, \|h\|_\infty \le 1}|\e h(W) - \e h(Z)|.
\end{eqnarray*}

It is generally much harder to obtain an optimal bound on the Kolmogorov distance than on the Wasserstein distance.
There is a discussion on this and examples of bounding the Wasserstein distance and the Kolmogorov distance are given in Chen \cite{Chena}.

We now state a proposition that concerns the boundedness properties of the solution $f_h$, given by (\ref{Stein-eqn-soln}), of the Stein equation (\ref{Stein-equation}) for $h$ either bounded or absolutely continuous with bounded $h'$.
The use of these boundedness properties is crucial for bounding the Wasserstein, Kolmogorov and total variation distances.
\begin{proposition} \label{prop-Stein-eqn-soln}
Let $f_h$ be the unique solution, given by (\ref{Stein-eqn-soln}), of the Stein equation (\ref{Stein-equation}), where $h$ is either bounded or absolutely continuous.

1. If $h$ is bounded, then
\begin{equation}\label{bd-bounded}
\|f_h\|_{\infty} \le \sqrt {\pi/2}\|h-\e h(Z)\|_{\infty}, ~~\|f'_h\|_{\infty} \le
2\|h-\e h(Z)\|_{\infty}.
\end{equation}

2. If $h$ is absolutely continuous with bounded $h'$, then
\begin{equation}\label{bd-abs-cont}
\|f_h\|_{\infty} \le 2\|h'\|_{\infty},~~\|f'_h\|_{\infty} \le \sqrt {2/\pi}\|h'\|_{\infty}, ~~\|f''_h\|_{\infty} \le 2\|h'\|_{\infty}.
\end{equation}

3. If $h = I_{(-\infty,x]}$ where $x \in \r$, then, writing $f_h$ as $f_x$,
\begin{equation}\label{bd-indicator-1}
0 < f_x(w) \le  \sqrt{2\pi}/4, ~~|wf_x(w)| \le 1, ~~|f_x'(w)| \le 1,
\end{equation}
and for all $w, u, v \in \r$,
\begin{equation} \label{bd-indicator-2}
|f_x'(w) - f_x'(v)| \le 1,
\end{equation}
\begin{equation} \label{indicator-bd-4}
|(w+u)f_x(w+u)-(w+v)f_x(w+v)|\le(|w|+\sqrt{2\pi}/4)(|u|+|v|).
\end{equation}

4. If $h = h_{x,\epsilon}$ where $x\in \r$, $\epsilon >0$, and
\begin{equation*}
h_{x,\epsilon}(w) =\left\{\begin{array}{ll}
1,&w\le x,\\
0,&w\ge x+\epsilon,\\
1+\epsilon^{-1}(x-w),&x<w<x+\epsilon,
\end{array}\right.
\end{equation*}
then, writing $f_h$ as $f_{x,\epsilon}$, we have for all $w,v, t\in \r$,
\begin{equation}\label{f_{x,epsilon}-bd-1}
0\le f_{x,\epsilon}(w) \le 1, \quad  |f'_{x,\epsilon}(w)|\le1,\quad  |f'_{x,\epsilon}(w) - f'_{x,\epsilon}(v)| \le 1
\end{equation}
and
\begin{eqnarray} \notag
&& |f'_{x,\epsilon}(w+t)-f'_{x,\epsilon}(w)|  \\ \notag
&& \qquad \le (|w|+1)|t|+ \frac{1}{\epsilon}\int_{t\wedge 0}^{t \vee 0}I(x \le w+u\le x+\epsilon)du \\ \label{f_{x,epsilon}-bd-2}
&& \qquad \le (|w|+1)|t| +  I(x - 0\vee t \le w \le x - 0\wedge t + \epsilon).
\end{eqnarray}
\end{proposition}

Except for (\ref{f_{x,epsilon}-bd-2}), which can be found on page 2010 in Chen and Shao \cite{CS}, the bounds in Proposition \ref{prop-Stein-eqn-soln} and their proofs can be found in Lemmas 2.3, 2.4 and 2.5 in Chen, Goldstein and Shao \cite{CGS}.

\subsection{Stein identities and error terms}

Let $W$ be a random variable with $\e W = 0$ and $\mathrm{Var}(W) = 1$.
In addition to using the boundedness properties of the solution $f_h$ of the Stein equation (\ref{Stein-equation}), we also need to exploit the probabilistic structure of $W$, in order to bound the error term in (\ref{error-term}).
This is done through the construction of a Stein identity for $W$ for normal approximation.
This is perhaps best understood by looking at a specific example.

Let $X_1, \ldots,X_n$ be independent random variables with $\e X_i = 0$, $\mathrm{Var}(X_i) =\sigma_i^2$ and $\e|X_i|^3 < \infty$.
Let $W = \sum_{i=1}^n X_i$ and $W^{(i)} = W - X_i$ for $i=1,\ldots,n$.
Assume that $\mathrm{Var}(W) = 1$, which implies $\sum_{i=1}^n \sigma_i^2 = 1$.
Let $f \in \mathcal{C}_B^2$.
Using the independence among the $X_i$ and the property that $\e X_i = 0$, we have
\begin{eqnarray} \notag
\e Wf(W) &=& \sum_{i=1}^n \e X_i f(W) =\sum_{i=1}^n \e X_i [f(W^{(i)}+X_i) - f(W^{(i)})] \\ \notag
&=& \sum_{i=1}^n\e\int_0^{X_i} X_i f'(W^{(i)}+t) dt \\ \notag
&=& \sum_{i=i}^n\e\int_{-\infty}^\infty f'(W^{(i)}+t)\hat{K}_i(t)dt \\ \label{Stein-identity-indep-1}
&=& \sum_{i=i}^n\e\int_{-\infty}^\infty f'(W^{(i)}+t)K_i(t)dt
\end{eqnarray}
where
\begin{eqnarray*}
\hat{K}_i(t) &=& X_i[I(X_i>t>0)-I(X_i<t<0)],\\
K_i(t) &=& \e\hat{K}_i(t).
\end{eqnarray*}
It can be shown that for each $i$, $\sigma_i^{-2}K_i$ is a probability density function.
So (\ref{Stein-identity-indep-1}) can be rewritten as
\begin{equation} \label{Stein-identity-indep-2}
\e Wf(W) = \sum_{i=1}^n \sigma_i^2\e f'(W^{(i)}+T_i)
\end{equation}
where $T_1,\ldots,T_n, X_1,\ldots,X_n$ are independent and $T_i$ has the density $\sigma_i^{-2}K_i$.
Both the equations (\ref{Stein-identity-indep-1}) and (\ref{Stein-identity-indep-2}) are Stein identities for $W$ for normal approximation.
From (\ref{Stein-identity-indep-2}) we obtain
\begin{eqnarray} \notag
\e [f'(W) - Wf(W)]&=& \sum_{i=1}^n \sigma_i^2\e[f'(W) - f'(W^{(i)})] \\ \label{Stein-identity-indep-3}
&& - \sum_{i=1}^n \sigma_i^2\e[f'(W^{(i)}+T_i) - f'(W^{(i)})]
\end{eqnarray}
where the error terms on the right hand side provide an expression for the deviation of $\e[f'(W) - Wf(W)]$ from $0$.
Now let $f$ be $f_h$ where $h \in \mathcal{C}^1$ such that $\|h'\|_\infty \le 1$.
Applying Taylor expansion to (\ref{Stein-identity-indep-3}), and using (\ref{d_W-h-C^1}) and (\ref{bd-abs-cont}), we obtain
\begin{eqnarray*}
d_W(W,Z) &\le& \sum_{i=1}^n \sup_{h\in \mathcal{C}^1, \|h'\|_\infty \le 1} \sigma_i^2\e|X_i| \|f_h''\|_\infty \\
&& +\sum_{i=1}^n \sup_{h\in \mathcal{C}^1, \|h'\|_\infty \le 1} \sigma_i^2\e|T_i|\|f_h''\|_\infty \\
&\le& 2\sum_{i=1}^n\sigma_i^2\e|X_i| +2\sum_{i=1}^n\sigma_i^2\e|T_i|.
\end{eqnarray*}
Since $\sigma_i^2\e|X_i| \le (\e|X_i|^3)^{2/3})(\e|X_i|^3)^{1/3} = \e|X_i|^3$ and $\sigma_i^2\e|T_i| = \sigma_i^2(1/2\sigma_i^2)\e|X_i|^3 = (1/2)\e|X_i|^3$, we have
\begin{equation} \label{bd-d_W-indep}
d_W(W,Z) \le 3\sum_{i=1}^n\e|X_i|^3.
\end{equation}
The bound in (\ref{bd-d_W-indep}) is of optimal order.
However, it is much harder to obtain a bound of optimal order for the Kolmogorov distance.
Proofs of such a bound on $d_K(W,Z)$ can be found in Chen \cite{Chena} and Chen, Goldstein and Shao \cite{CGS}.
Stein's identities for locally dependent random variables and other dependent random variables can be found in Chen and Shao \cite{CS} and Chen, Goldstein and Shao \cite{CGS}.

\subsection{Integration by parts}

Let $W$ be a random variable with $\e W = 0$ and $\mathrm{Var}(W) = 1$.
In many situations of normal approximation, the Stein identity for $W$ takes the form
\begin{equation} \label{Stein-id-T}
\e Wf(W) = \e Tf'(W)
\end{equation}
where $T$ is random variable defined on the same probability as $W$ such that $\e|T| < \infty$, and $f$ an absolutely continuous function for which the expectations exist.
Typical examples of such situations are cases when $W$ is a functional of Gaussian random variables or of a Gaussian process.
In such situations, the Stein identity (\ref{Stein-id-T}) is often constructed using integration by parts as in the case of the Stein identity for $N(0,1)$.

By letting $f(w) = w$, we obtain $\e T = \e W^2 = 1$.
Let $\scrF$ be a $\sigma$-algebra with respect to which $W$ is measurable.
From (\ref{Stein-id-T}),
\begin{eqnarray*}
\e[f'(W) - Wf(W)] &=& \e[f'(W)(1-T)] \\
&=& \e[f'(W)\e(1-T\big|\scrF)].
\end{eqnarray*}
Now let $f = f_h$ where $h \in \scrG$, a separating class of functions.
Assume that $||f'_h||_\infty < \infty$.
Then by ({\ref{def-d_G}),
\begin{eqnarray} \notag
d_\scrG(W,Z) &=& \sup_{h\in\scrG}|\e[f'_h(W) - Wf_h(W)]| \\ \notag
&\le& \left(\sup_{h\in\scrG}\|f'_h\|_\infty\right)\e|\e(1-T\big|\scrF)| \\ \notag
&\le& \left(\sup_{h\in\scrG}\|f'_h\|_\infty\right)\sqrt{\mathrm{Var}(\e(T\big|\scrF))},
\end{eqnarray}
where for the last inequality it is assumed that $\e(T\big|\scrF)$ is square integrable.
By Proposition \ref{prop-Stein-eqn-soln},
\begin{equation}
\sup_{h\in\scrG}\|f'_h\|_\infty \le \left\{\begin{array}{ll}
\sqrt{2/\pi}, & \quad \scrG = \scrG_W, \\  \label{bd-f'_h-G}
1, & \quad \scrG = \scrG_K,\\
2, & \quad \scrG = \scrG_{TV}.
\end{array}\right.
\end{equation}
This implies
\begin{equation} \label{bd-d_WKTV}
d(W,Z) \le \theta \e|\e(1-T\big|\scrF)| \le \theta \sqrt{\mathrm{Var}(\e(T\big|\scrF))}
\end{equation}
where
\begin{equation}
\theta = \left\{\begin{array}{ll}
\sqrt{2/\pi}, & \quad d(W,Z) = d_W(W,Z), \\  \label{bd-d_WKTC-theta}
1, & \quad d(W,Z) = d_K(W,Z),\\
2, & \quad d(W,Z) = d_{TV}(W,Z).
\end{array}\right.
\end{equation}

For the rest of this section, we will present two approaches to the construction of the Stein identity (\ref{Stein-id-T}).

Let $Z \sim N(0,1)$ and let $\psi$ be an absolutely continuous function such that $\e \psi(Z) = 0$, $\mathrm{Var}(\psi(Z)) =1$ and $\e \psi'(Z)^2 < \infty$.
Define $W=\psi(Z)$.
Following Chatterjee \cite{Chat}, we use Gaussian interpolation to construct a Stein identity for $\psi(Z)$.
Let $f$ be absolutely continuous with bounded derivative, which implies $|f(w)| \le C(1+|w|)$ for some $C >0$.
Let $Z'$ be an independent copy of $Z$ and let $W_t = \psi(\sqrt{t}Z + \sqrt{1-t}Z')$ for $0\le t \le 1$.
Then we have
\begin{eqnarray} \notag
\hspace*{-10pt}\e Wf(W)&=& \e(W_1 - W_0)f(W) = \e\int_0^1 f(W)\frac{\partial W_t}{\partial t}dt \\ \label{eq-W_t}
&=& \e\int_0^1 f(W)\left(\frac{Z}{2\sqrt{t}} - \frac{Z'}{2\sqrt{1-t}}\right)\psi'(\sqrt{t}Z + \sqrt{1-t}Z')dt. \qquad
\end{eqnarray}
Let
\begin{eqnarray*}
U_t &=& \sqrt{t}Z+\sqrt{1-t}Z', \\
V_t &=& \sqrt{1-t}Z-\sqrt{t}Z'.
\end{eqnarray*}
Then $U_t \sim N(0,1)$, $V_t \sim N(0,1)$, and $U_t$ and $V_t$ are independent.
This together with $|f(w)|\le C(1+|w|)$ implies the integrability of the right hand side of (\ref{eq-W_t}).
Solving for $Z$, we obtain
$$
Z = \sqrt{t}U_t + \sqrt{1-t}V_t.
$$
Equation (\ref{eq-W_t}) can be rewritten as
\begin{eqnarray*}
\e Wf(W) &=& \frac{1}{2}\int_0^1\frac{1}{\sqrt{t(1-t)}}\e f(\psi(\sqrt{t}U_t+\sqrt{1-t}V_t))V_t \psi'(U_t) \\
&=& \frac{1}{2}\int_0^1\frac{1}{\sqrt{t}} \e f'(W)\psi'(Z)\psi'(U_t)dt \\
&=& \frac{1}{2}\e\left[f'(W)\e\left(\int_0^1\frac{1}{\sqrt{t}}\psi'(Z)\psi'
(\sqrt{t}Z+\sqrt{1-t}Z')dt\big| Z\right)\right]
\end{eqnarray*}
where for the second equality we used the indepedence of $U_t$ and $V_t$, and applied the characterization equation for $N(0,1)$ to $V_t$.
Note that the characterizaton equation is obtained by integration by parts.
Hence we have for absolutely continuous $f$ with bounded derivative,
\begin{equation} \label{eq-T(Z)}
\e Wf(W) = \e T(Z)f'(W)
\end{equation}
where
\begin{equation} \label{def-T(x)}
T(x) = \int_0^1 \frac{1}{2\sqrt{t}}\e \left[\psi'(x)\psi'(\sqrt{t}x+\sqrt{1-t}Z')\right]dt.
\end{equation}
In \cite{Chat}, Chatterjee obtained a multivariate version of (\ref{eq-T(Z)}) where $\psi: \r^d \longrightarrow \r$ for $d \ge 1$.

Here is a simple application of (\ref{bd-d_WKTV}) and (\ref{eq-T(Z)}).
Let $X_1,\ldots,X_n$ be independent and identically distributed as $N(0,1)$.
Let
$$
W = \frac{\sum_{i=1}^n (X_i^2 - 1)}{\sqrt{2n}}.
$$
The random variable $W$ has the standardized $\chi^2$ distribution with $n$ degrees of freedom.
Let $\psi(X_i) = {\displaystyle \frac{X_i^2 - 1}{\sqrt{2n}}}$.
Then ${\displaystyle W = \sum_{i=1}^n \psi(X_i)}$.
Let $W^{(i)}= W - \psi(X_i)$ and let $X_i',\ldots,X_n'$ be an independent copy of $X_1,\ldots,X_n$. By the independence of $X_1,\ldots,X_n$, and by (\ref{eq-T(Z)}) and (\ref{def-T(x)}),
\begin{eqnarray*}
\e Wf(W) &=& \sum_{i=1}^n \e \psi(X_i)f'(W^{(i)}+g(X_i)) = \sum_{i=1}^n \e T(X_i)f'(W) \\
&=& \e Tf'(W)
\end{eqnarray*}
where
\begin{eqnarray*}
T &=& \sum_{i=1}^n T(X_i), \\
T(X_i) &=& \int_0^1 \frac{1}{2\sqrt{t}}\e \left(\psi'(X_i)\psi'(\sqrt{t}X_i+\sqrt{1-t}X_i')\big | X_i\right)dt \\
&=& \int_0^1 \frac{1}{n\sqrt{t}}\e\left(X_i(\sqrt{t}X_i+\sqrt{1-t}X_i')\big|X_i\right)dt = \frac{X_i^2}{n}.
\end{eqnarray*}
Therefore
$$
\mathrm{Var}(T) = \sum_{i=1}^n\mathrm{Var}\left(\frac{X_i^2}{n}\right) = \frac{2}{n}.
$$
By (\ref{bd-d_WKTV}) and since $\mathrm{Var}(\e(T\big|\scrF)) \le \mathrm{Var}(T)$, it follows that
\begin{equation} \label{bd-d_WKTV-sum-g}
d_W(W,Z) \le \frac{2}{\sqrt{\pi}}\frac{1}{\sqrt{n}},  \quad d_K(W,Z) \le \frac{\sqrt{2}}{\sqrt{n}}, \quad d_{TV}(W,Z) \le \frac{2\sqrt{2}}{\sqrt{n}}.
\end{equation}
Since $d_K(W,Z) \le d_{TV}(W,Z)$, this result is stronger and more general than what can be deduced from the Berry-Esseen theorem, which yields only the Kolmogorov bound.

We now present another approach to the construction of the Stein identity (\ref{eq-T(Z)}).   Consider $\calL^2 = L^2\left(\r, \frac{1}{\sqrt{2\pi}}e^{- \frac{x^2}{2}}dx\right)$ endowed with the inner product,
$$
\langle f,g\rangle =\int_\r f(x)g(x)\frac{1}{\sqrt{2\pi}}e^{-\frac{x^2}{2}}dx.
$$
Let $L$ be the Ornstein-Uhlenbeck operator defined on $\calD^2 \subset \calL^2$, where
$$
\calD^2 = \{g: g' ~\text{is absolutely continuous}, ~ g, g', g'' \in \mathcal{L}^2\}.
$$
That is,
$$
L = \frac{d^2}{d^2x} - x\frac{d}{dx}.
$$
Let $f \in \calL^2$ be absolutely continuous such that $f' \in \calL^2$ and let $g \in \text{Dom}(L)$.
Assume that $g'(x)f(x)e^{-x^2/2}\rightarrow 0$ as $x \rightarrow \pm$.
Then we have the integration by parts formula,
\begin{equation} \label{int-by-parts}
\langle Lg,f\rangle = - \langle g', f'\rangle.
\end{equation}
 Let $\psi \in \calL^2$ be absolutely continuous such that $\e \psi(Z) =0$, $\mathrm{Var}(\psi(Z))=1$ and $\psi'$ is bounded by a polynomial.
 A solution $g_\psi$ of the equation
$$
Lg = \psi
$$
is given by
\begin{eqnarray*}
g_\psi(x) &=& L^{-1}\psi = - \int_0^\infty P_t\psi dt \\
&=& - \int_0^\infty \e \psi(e^{-t}x +\sqrt{1-e^{-2t}}Z)dt
\end{eqnarray*}
where $L^{-1}$ is a pseudo-inverse of $L$, $(P_t)_{t\ge0}$ is the Ornstein-Uhlenbeck semigroup defined on $\calL^2$, and $Z$ and $Z'$ are independent, each distributed as $N(0,1)$.
By the integration by parts formula (\ref{int-by-parts}), we have for absolutely continuous $f$ such that $\|f'\|_\infty < \infty$,
\begin{eqnarray} \notag
&& \e \psi(Z)f(\psi(Z))  \\  \notag
&& \quad = \langle \psi,f(\psi)\rangle = \langle LL^{-1}\psi,f(\psi)\rangle \\  \notag
&& \quad =- \langle (L^{-1}\psi)', f'(\psi)\psi'\rangle  = - \langle g_\psi',f'(\psi)\psi'\rangle  \\  \notag
&& \quad = \e \int_0^\infty e^{-s}f'(\psi(Z))\psi'(Z)\e \left(\psi'(e^{-s}x + \sqrt{1-e^{-2s}}Z')\big|Z\right) ds  \\ \label{eq-T(Z)-2}
&& \quad = \e\int_0^1 \frac{1}{2\sqrt{t}}f'(\psi(Z))\psi'(Z)\e\left(\psi'(\sqrt{t}Z + \sqrt{1-t}Z')\big|Z\right)dt
\end{eqnarray}
where the last equality follows from the change of variable, $t = e^{-2s}$.
By using the fact that polynomials are dense in $\calL^2$, (\ref{eq-T(Z)-2}) can be shown to hold for $\e \psi'(Z)^2 < \infty$.
By letting $W = \psi(Z)$, the identity (\ref{eq-T(Z)-2}) is indeed the same as (\ref{eq-T(Z)}).

This approach of using the integration by parts formula (\ref{int-by-parts}) to construct a Stein identity is a special case of that of Nourdin and Peccati \cite{NP}, who considered $L^2(\Omega) = L^2(\Omega, \scrF, P)$, where $\scrF$ is the complete $\sigma$-algebra generated by a standard Brownian motion defined on $\Omega$.
The integration by parts formula in this setting is
\begin{equation} \label{ibp-Omega}
\langle LF, G \rangle_{L^2(\Omega) }= - \e[\langle DF,DG \rangle_{L^2(\r_+,dx)}]
\end{equation}
where $F\in \mathbb{D}^{2,2}$, $G \in \mathbb{D}^{1,2}$, $L$ is the Ornstein-Uhlenbeck operator defined on $\mathbb{D}^{2,2} \subset L^2(\Omega)$ and $D$ the Malliavin derivative with the domain $\mathbb{D}^{1,2}$.

In this chapter, we will use a similar integration by parts formula in white noise analysis involving the Hida derivative to obtain a general error bound in the normal approximation for white noise functionals.
In the process, it is found necessary to extend the domain of the Hida derivative to allow application of the integration by parts formula to the normal approximation.

\section{Hida Distributions}

In this and the next sections, we will give a brief description of Hida's white noise calculus based on Lee's reformulation on the abstract Wiener space $(\calS_0,\,\calS_{-p})$ for $p>\frac{1}{2}$.
For more details, see \cite{{La},{Lb},{Lc},{Ld}}.

\subsection{White noise space}

Let ${\bf A}=-(d/dt)^2+1+t^2$ be a densely defined self-adjoint operator on the $L^2$-space $L^2(\r,dt)$ with respect to the Lebesgue measure $dt$, and $\{h_n;\,n\in\n_0\equiv\n\cup\{0\})\}$ be a complete orthonormal set (CONS for abbreviation) for $L^2(\r,dt)$, consisting of all Hermite functions on $\r$, formed by the eigenfunctions of $\bf A$ with corresponding eigenvalues $2n+2$, $n\in\n_0$, where
$$
h_n(t)=\frac{1}{\sqrt{\sqrt{\pi}2^nn!}}\,H_n(t)\,e^{-t^2/2},
$$
$H_n(t)=(-1)^ne^{t^2}\frac{d^n}{dt^n}e^{-t^2}$ being the Hermite polynomial of the degree $n$. \par

Let $\calS$ be the Schwartz space of real-valued, rapidly decreasing, and infinitely differentiable functions on $\r$ with its dual $\calS'$, the spaces of tempered distributions.
For each $p\in\r$, let $\calS_p$ denote the space of all functions $f$ in $\calS'$ satisfying the condition that
$$
|f|_p^2\equiv\sum_{n=0}^\infty\,(2n+2)^{2p}\,|( f,\,h_n) |^2<+\infty,
$$
where $(\cdot,\cdot)$ always denotes the $\calS'$-$\calS$ pairing from now on.
Then $\calS_p$, $p\in\r$, forms a real Hilbert space with the inner product $\l\cdot,\cdot\g_p$ induced  by $|\cdot|_p$.
The dual space $\calS_p'$, $p\in\r$, is unitarily equivalent to $\calS_{-p}$.
Applying the Riesz representation theorem, we have the continuous inclusions:
$$
\calS\subset \calS_q\subset \calS_p\subset L^2(\r,dt)= \calS_0\subset \calS_{-p}\subset \calS_{-q}\subset \calS',
$$
where $0<p<q<+\infty$, $\calS$ is the projective limit of $\{\calS_p;\,p>0\}$.
In fact, $\calS$ is a nuclear space, and thus $\calS'$ is the inductive limit of $\{\calS_{-p};\, p>0\}$.

A well known fact is that the Minlos theorem (see \cite{GV}) guarantees the existence of the white noise measure $\mu$ on $(\calS',\scrB(\calS'))$, $\scrB(\calS')$ being the Borel $\sigma$-field of $\calS'$, the characteristic functional of which is given by
\begin{equation}\label{2.1}
\int_{\calS'}e^{\I\,(x,\,\eta)}\,\mu(dx)=\,e^{-\frac{1}{2}|\eta|_0^2},~~~\mbox{for all $\eta\in\calS$},
\end{equation}
where $\I=\sqrt{-1}$.
One can easily show that the measurable support of $\mu$ is contained in $\calS_{-p}$ and $\mu$ coincides with the Wiener measure on the abstract Wiener space $(\calS_0,\,\calS_{-p})$ for $p>\frac{1}{2}$.

As a random variable on $(\calS',\scrB(\calS'),\mu)$, $(\cdot,\,\eta)$ has the normal distribution with mean $0$ and variance $|\eta|_0^2$ for any $\eta\in\calS$.
For each $\rho\in \calS_0$, choose a sequence $\{\eta_n\}\subset\calS$ so that $\eta_n\to\rho$ in $\calS_0$.
Then it follows from (\ref{2.1}) that $\{(\cdot,\,\eta_n)\}$ forms a Cauchy sequence in $(L^2)\equiv L^2(\calS',\mu)$, the $L^2$-space of all complex-valued square-integrable functionals on $\calS'$ with respect to $\mu$.
Denote by $\l\cdot,\,\rho\g$ the $L^2$-limit of $\{(\cdot,\,\eta_n)\}$.
Then $\l\cdot,\,\rho\g\sim\calN(0,\,|\rho|_0^2)$ for $\rho\in \calS_0$.
Consequently, the Brownian motion $B=\{B(t);\,t\in\r\}$ on $(\calS',\calB(\calS'),\mu)$ can be represented by
$$
   B(t\,;\,x)=\left\{
   \begin{array}{rl}
   \l x,\,{1}_{[0,\,t]}\g,\quad &\makebox{if $t\geq 0$}\\
   -\l x,\,{1}_{[t,\,0]}\g,\quad &\makebox{if $t<0$},~x\in\calS'.
\end{array}\right.
$$
Taking the time derivative formally, we get $\dot{B}(t\,;\,x)=\,x(t)$, $x\in\calS'$.
Thus an element $x\in\calS'$ is viewed as a sample path of white noise $\dot{B}(t,\,x)$ and the space $(\calS',\,\mu)$ is referred to as a white noise space.

\subsection{The $S$-transform}

The $S$-transform $S\varphi$ of $\varphi\in (L^2)$ is a Bargmann-Segal analytic functional on $\scrC\calS_0$ given by
$$
   S\varphi(\eta)=\,e^{-\frac{1}{2}\int_{-\infty}^\infty \eta(t)^2\,dt}\int_{\calS'}\varphi(x)\,e^{\l x,\,\eta\g}\,\mu(dx),~~~\eta\in\scrC\calS_0.
$$
We should note that $S\varphi(\eta)=\mu\varphi(\eta)$ for $\eta\in\calS_0$, where $\mu\varphi=\mu\ast\varphi$, the convolution of $\mu$ and $\varphi$.

Let $\varphi\in (L^2)$ be given.
Then it follows from \cite{{La},{Lb}} that $D^nS\varphi(0)$ is a symmetric $n$-linear operator of Hilbert-Schmidt type on $\scrC\calS_0\times\cdots\times\scrC\calS_0$, where $D$ is the Fr\'echet derivative of $S\varphi$.
Also, it admits the Wiener-It\^o decomposition
$$
 \varphi(x)=\,\sum_{n=0}^\infty\,\frac{1}{n!}:D^nS\varphi(0)x^n:,
$$
and
$$
\|\varphi\|_{2,0}^2\equiv\int_{\calS'}|\varphi(x)|^2\,\mu(dx)=~\sum_{n=0}^\infty\,\frac{1}{n!}\,
\|D^nS\varphi(0)\|^2_{\calH\calS^n(\scrC\calS_0)},
$$
where $\|\cdot\|_{\calH\calS^n(K)}$ denotes the Hilbert-Schmidt operator norm of a $n$-linear functional on a Hilbert space $K$, and for any symmetric $n$-linear Hilbert-Schmidt operator $T$ on $\scrC\calS_0\times\cdots\times\scrC\calS_0$,
\begin{align*}
:Tx^n\!:\,\equiv&\int_{\calS'}T(x+\I\,y)^n\,\mu(dy)\\
=&\sum_{j_1,\ldots,j_n=0}^\infty T(h_{j_1},\ldots,h_{j_n})\int_{\calS'}\prod_{k=1}^n\,(x+\I\,y,\,h_{j_k})\,\mu(dy),
\end{align*}
which is in $(L^2)$ with $\|:Tx^n\!:\|_{2,0}=\sqrt{n!}\,\| T\|_{\calH\calS^n(\scrC\calS_0)}$ (see \cite{La}).
In fact, the $S$-transform is a unitary operator from $(L^2)$ onto the Bargmann-Segal-Dwyer space $\calF^1(\scrC\calS_0)$ over $\scrC\calS_0$ (see \cite{Ld}).

\begin{remark}
Let $H$ be a complex Hilbert space. For $r>0$, denote by $\calF^r(H)$ the class of analytic functionals on $H$ with norm $\|\cdot\|_{\calF^r(H)}$ such that
$$
\|f\|_{\calF^r(H)}^2=~\sum_{n=0}^\infty\,\frac{r^n}{n!}\,\|D^nf(0)\|_{\calH\calS^n(H)}^2<+\infty,
$$
called the Bargmann-Segal-Dwyer space.
Members of $\calF^r(H)$ are called Bargmann-Segal analytic functionals.
See also \cite{Ld}.
\end{remark}

\subsection{Test and generalized white noise functionals}

To study nonlinear functionals of white noise, Hida originally established a test-generalized functions setting $(L^2)^+\subset(L^2)\subset(L^2)^-$ (see \cite{{Hida},{HidaSi},{Si}}).
After Hida, Kubo and Takenaka \cite{KT} reformulated Hida's theory by taking different setting $(\calS)\subset (L^2)\subset (\calS)'$.
The space $(\calS)$ is an infinite dimensional analogue of the Schwartz space $\calS$ on $\r$.
We briefly describe as follows.

For $p\in\r$ and $\varphi\in (L^2)$, define
\begin{equation}\label{2.2}
\|\varphi\|_{2,p}^2=\sum_{n=0}^\infty\,\frac{1}{n!}\,\|D^nS\varphi(0)\|^2_{\calH\calS^n(\calS_{-p})}.
\end{equation}
and let $(\calS_p)$ be the completion of the collection $\{\varphi\in (L^2);\,\|\varphi\|_{2,p}<+\infty\}$ with respect to $\|\cdot\|_{2,p}$-norm.
Then $(\calS_p)$, $p\in\r$, is a Hilbert space with the inner product induced by $\|\cdot\|_{2,p}$-norm.
For $p,q\in\r$ with $q\geq p$, $(\calS_q)\subset(\calS_p)$ and the embedding $(\calS_q)\hookrightarrow(\calS_p)$ is of Hilbert-Schmidt type, whenever $q-p>1/2$.
Set $(\calS)=\bigcap_{p>0}\,(\calS_p)$ endowed with the projective limit topology.
Then $(\calS)$ is a nuclear space and will serve as the space of test white noise functionals.
The dual $(\calS)'$ of $(\calS)$ is the space of generalized white noise functionals (or often called Hida distributions).
By identifying the dual $(\calS_p)'$ of $(\calS_p)$, $p>0$, with $(\calS_{-p})$, we have a Gel'fand triple $(\calS)\subset (L^2)\subset(\calS)'$ and the continuous inclusions: for $p\geq q> 0$,
$$
 (\calS)\subset(\calS_p)\subset(\calS_q)\subset (L^2)\subset(\calS_{-q})\subset(\calS_{-p})\subset(\calS)',
$$
where $(\calS)'$ is the inductive limit  of the $(\calS_{-p})$, $p>0$.
Hereafter, the dual pairing of $(\calS)'$ and $(\calS)$ will be denoted by $\ll\cdot,\cdot\gg$.
One notes that $(\calS_p)$, $p\geq 0$, is the domain of the second quantization $\Gamma({\bf A}^p)$ of ${{\bf A}^p}$ and $(\calS_0)=(L^2)$.

For $\varphi\in (\calS_p)$ with $p>0$, it is natural to extend the domain of $D^nS\varphi(0)$ to $\scrC\calS_{-p}\times\cdots\times\scrC\calS_{-p}$ and then define
$$
S\varphi(z)=\sum_{n=0}^\infty\,\frac{1}{n!}\,D^nS\varphi(0)z^n,~~z\in\scrC\calS_{-p}.
$$
It is clear that
$$
\big|S\varphi(z)\big|\leq\,\|\varphi\|_{2,p}\cdot e^{\frac{1}{2}|z|_{-p}^2},~~~z\in\scrC\calS_{-p}.
$$

On the other hand, by directly computing (\ref{2.2}), it is easy to see that $e^{(\cdot,\,\eta)}\in (\calS_p)$ for any $\eta\in\scrC\calS_p$, $p>0$.
We then extend the $S$-transform to a function $F\in(\calS_{-p})$, $p>0$, by setting
$$
SF(\eta)=\,e^{-\frac{1}{2}\int_{-\infty}^\infty \eta(t)^2\,dt}\ll F,\,e^{(\cdot,\,\eta)}\gg,~~~\eta\in\scrC\calS_p,
$$
where $\|e^{-\frac{1}{2}\int_{-\infty}^\infty \eta(t)^2\,dt+(\cdot,\,\eta)} \|_{2,p}=e^{\frac{1}{2}|\eta|_p^2}$.
In fact, the $S$-transform is a unitary operator from $(\calS_p)$ onto $\calF^1(\scrC\calS_{-p})$ for any $p\in\r$ (see \cite{Ld}).
In other words, for any $F\in (\calS_p)$ with $p\in\r$,
\begin{equation*}
\| F\|_{2,p}^2=\sum_{n=0}^\infty\,\frac{1}{n!}\,\|D^nSF(0)\|^2_{\calH\calS^n(\calS_{-p})}.
\end{equation*}

\begin{remark}
The image $SF$, $F\in (\calS)'$, is also called the $U$-functional associated with $F$.
Hida then studied the white noise calculus of generalized white noise functionals through their $U$-functionals.
See \cite{{Hida},{HKPS},{HidaSi},{Si}}.
\end{remark}

\vskip 0.1cm

\noindent
{\it$-$~Analytic version of $(\calS)$}

\vskip 0.2cm

For $p\in\r$, denote by $\calA_p$ the space of analytic functions $f$ defined on $\scrC\calS_{-p}$ satisfying the exponential growth condition:
$$
\| f\|_{\calA_p}\equiv\sup\{| f(z)|e^{-\frac{1}{2}|z|_{-p}^2};\,z\in\scrC\calS_{-p}\}<+\infty.
$$
Then $(\calA_p,\,\|\cdot\|_{\calA_p})$ is a Banach space and, by restriction, $\calA_p$ is continuously embedded in $\calA_q$ for $p>q$.
Set $\calA_\infty=\bigcap_{p\in\r}\,\calA_p$ endowed with the projective limit topology.
Then $\calA_\infty$ becomes a locally convex topological algebra.

For any $\varphi\in (\calS_p)$ with $p>\frac{1}{2}$, let
$$
\tilde{\varphi}(z)=\,\sum_{n=0}^\infty\,\frac{1}{n!}\int_{\calS_{-p}}D^nS\varphi(0)(z+\I\,y)^n\,
\mu(dy),~~z\in\scrC\calS_{-p},
$$
which is well-defined since ${\rm supp}\,(\mu)\subset\calS_{-p}$ with $p>\frac{1}{2}$.
One notes that the above sum converges absolutely and uniformly on each bounded set in $\scrC\calS_{-p}$.
In addition, $\varphi=\tilde{\varphi}$ almost all in $\calS'$ with respect to $\mu$.

\vskip 0.2cm

\begin{theorem} \label{thm-lee} \cite{{Ld},{Lf},{Le}}~\
\begin{enumerate}[(iii)]
\item[$(i)$]
Let $p>1,\,r>1$ and $s>\frac{1}{2}$.
Then, for any $\varphi\in (\calS_p)$, $\tilde{\varphi}$ is analytic on $\scrC\calS_{-p}$, and there are two constants $\alpha_p>0$ and $\beta_{p,r}>0$ such that
$$
 \alpha_p\cdot\|\tilde{\varphi}\|_{\calA_{p-s}}\leq\,\| \varphi\|_{2,p}\leq\,\beta_{p,r}\cdot\|\tilde{\varphi}\|_{\calA_{p+r}}.
$$

\item[$(ii)$]
Let $p>0$ and $r>\frac{1}{2}$.
Then, for any $\varphi\in (\calS_{p+r})$,
$$
  C_r^{-1}\|\varphi\|_{2,p}\leq\,\|S\varphi\|_{\calA_{p+r}}\leq\,\|\varphi\|_{2,p+r},
$$
where $C_r^{3/2}=\int_{\calS'}e^{|x|_{-r}^2}\,\mu(dx)$.

\item[$(iii)$]
Let $p\in\r$.
There exists $\alpha_p>0$ such that
$$
   \alpha_p\cdot\| f\|_{\calF^1(\scrC\calS_{-p+2})}\leq\,\| f\|_{\calA_p}\leq \,\| f\|_{\calF^1(\scrC\calS_{-p})}
$$
for any analytic function $f$ on $\scrC\calS_{-p}$.
\end{enumerate}
\end{theorem}

\begin{corollary}\cite{Ld}\
\begin{enumerate}[(iii)]
\item[$(i)$]
For $\varphi,\psi\in (\calS_p)$ with $p>1$, $\varphi(x)=\psi(x)$ $\mu$-a.e. $x$ in $\calS'$ if and only if $\tilde{\varphi}(z)=\tilde{\psi}(z)$ for all $z\in\scrC\calS_{-p}$.

\item[$(ii)$]
$\calA_\infty|_{\calS'}\equiv\{\varphi|_{\calS'};\,\varphi\in\calA_\infty\}\subset (\calS)$ and $\calA_\infty=(\calS\tilde{)}\equiv\{\tilde{\varphi};\,\varphi\in (\calS)\}$.

\item[$(iii)$]
The families of norms $\{\|\cdot\|_{\calA_p};\,p>0\}$ and $\{\|\cdot\|_{2,p};\,p>0\}$ are equivalent in $\calA_\infty$.
\end{enumerate}
\end{corollary}

Recall that any function $f\in (L^2)$ is identified with the equivalent class of functions in which any function is equal to $f$ almost all with respect to $\mu$.
In this sense, we identify $(\calS)$ with $\calA_\infty$, and call $\calA_\infty$ the analytic version of $(\calS)$. It is noted that, for any $F\in (\calS)'$ and $\varphi\in(\calS)$, $\ll F,\,\varphi\gg=\ll F,\,\tilde{\varphi}\gg$.

\section{ Hida Derivatives}

For $F\in (\calS)'$ and $\eta\in\calS$, define $\partial_\eta F\in (\calS)'$ by
\begin{equation}\label{3.1}
\ll \partial_\eta F,\,\varphi\gg =\,\ll F,\,\tilde{\eta}\varphi\gg-\ll F,\,D_\eta\varphi\gg,~~~\varphi\in (\calS),
\end{equation}
where $\tilde{\eta}(x)=(x,\,\eta)$, $x\in\calS'$, and $D_\eta\varphi$ is the G\^ateaux derivative of $\varphi$ in the direction of $\eta$.

\begin{remark}
By using the chain rule and applying the integration by parts formula given in Theorem \ref{thm-int-by-parts},
\begin{align*}
\ll D_\eta\varphi,\,\psi\gg=&\int_{\calS'}D_\eta(\varphi\psi)(x)\,\mu(dx)-
\int_{\calS'}\varphi(x)\,D_\eta\psi(x)\,\mu(dx)\\
=&\int_{\calS'}(x,\,\eta)\,\varphi(x)\,\psi(x)\,\mu(dx)-
\int_{\calS'}\varphi(x)\,D_\eta\psi(x)\,\mu(dx)
\end{align*}
for any $\varphi,\,\psi\in (\calS)$ and $\eta\in\calS$.
This implies that $\partial_\eta \varphi=D_\eta \varphi$.
\end{remark}

Let $F\in (\calS_p)$, $p\in\r$, and $\eta\in\calS$.
Putting $\varphi=\,e^{(\cdot,\,h)-\frac{1}{2}\int_{-\infty}^\infty h(t)^2\,dt}$, $h\in\scrC\calS$, in (\ref{3.1}) and applying the Cauchy integral formula,
\begin{align}
S\partial_\eta F(h)=&\left.\frac{d}{dw}\right|_{w=0}SF(h+w\,\eta)\nonumber\\
=&\,\frac{1}{2\pi i}\int_{|w|=\frac{1}{2}|\eta|_{-p}^{-1}}\frac{SF(h+w\,\eta)}{w^2}\,dw,\label{3.2}
\end{align}
from which it follows that
\begin{equation}\label{3.3}
\bigl|S\partial_\eta F(h)\bigr|\leq\,{\rm Const.}\,\|F\|_{2,p}\cdot|\eta|_{-p}\cdot e^{\frac{1}{2}|h|_{-p+1}^2}.
\end{equation}

By virtue of (\ref{3.2}), (\ref{3.3}) and Theorem \ref{thm-lee}, we can conclude the following facts:
For $F\in (\calS_p)$ with $p\in\r$, $\partial_\eta F$ in (\ref{3.2}) can be extended to $\eta\in\scrC\calS_{-p}$ by defining
$$
\partial_\eta F=S^{-1}\biggl(\left.\frac{d}{dw}\right|_{w=0}SF(\cdot+w\eta)\biggr).
$$
Moreover, by applying the characterization theorem in \cite{LS} (see also \cite{PS}),
\begin{equation}\label{3.4}
\|\partial_\eta F\|_{2, p-3}\leq\,{\rm Const.}\,\| F\|_{2,p}\cdot |\eta|_{-p},
\end{equation}
where such a constant is independent of the choice of $p, F,\eta$.
It is noted that $\partial_\eta\,\varphi=\partial_\eta\,\tilde{\varphi}$ for $\varphi\in (\calS_p)$ with $p>\frac{1}{2}$.

By (\ref{3.4}), the mapping $(\eta,\,\varphi)\to\,\ll\partial_\eta F,\,\varphi\gg$, $F\in (\calS)'$, is bilinear and continuous from $(\calS)\times\scrC\calS$ into $\c$.
By applying the kernel theorem (see \cite{BK}), there exists a unique element $K_{F}\in\scrC\calS'\otimes (\calS)'$ such that
\begin{equation*}
 \ll\partial_\eta F,\,\varphi\gg=~(\!(K_{F},\,\eta\otimes\varphi)\!),
\end{equation*}
where $(\!(\cdot,\cdot)\!)$ is the $\scrC\calS'\otimes(\calS)'$--$\scrC\calS\otimes (\calS)$ pairing.
Symbolically, we express such an identity by the formal integral as follows:
\begin{equation}\label{3.5}
\ll\partial_\eta\, F,\,\varphi\gg=~\int_{-\infty}^\infty\int_{\calS'}\,K_F(t;\,x)\,\eta(t)\,\varphi(x)\,\mu(dx)dt.
\end{equation}

\vskip 0.1cm

\noindent
{\it$-$~$\dot{B}(t)$-differentiation}

\vskip 0.2cm

Assume that $\frac{\ln 6}{2\ln 2}<q<p$.
Let $\varphi\in (\calS_p)$ and $\eta\in\scrC\calS_{-p}$. Observe that
\begin{align}
&\left.\frac{d}{dw}\right|_{w=0}S\varphi(z+w\eta)\nonumber\\
&~~~=\left.\frac{d}{dw}\right|_{w=0}\sum_{n=0}^\infty\,
\frac{1}{n!}\,D^nS\varphi(0)(z+w\eta)^n\nonumber\\
&~~~=\sum_{n=1}^\infty\,\frac{1}{(n-1)!}\,D^nS\varphi(0)z^{n-1}\eta\nonumber\\
&~~~=\sum_{n=1}^\infty\,\frac{1}{(n-1)!}\,\int_{\calS_{-q}}
\int_{\calS_{-q}}D^nS\varphi(0)(x+z+\I\,y)^{n-1}\eta\,\,\mu(dy)\mu(dx),\label{3.6}
\end{align}
 and
\begin{align}
&\sum_{n=1}^\infty\,\frac{1}{(n-1)!}\,\left|\int_{\calS_{-q}}D^n
S\varphi(0)(x+z+\I\,y)^{n-1}\eta\,\,\mu(dy)\right| \nonumber\\
&~~~\leq\left\{\sum_{n=1}^\infty\,\frac{1}{(n-1)!}\,\| D^nS\varphi(0)\|_{\calH\calS^n(\scrC\calS_{-q})}^2\right\}^{\!\frac{1}{2}}\!\!\cdot\! \int_{\calS'}e^{3|y|_{-q}^2}\,\mu(dy)\nonumber\\
&~~~~~~~~~\times |\eta|_{-q}\,e^{\frac{3}{2}|z|_{-q}^2}e^{\frac{3}{2}|x|_{-q}^2}\nonumber\\
&~~~\leq\,\omega_{p-q}\cdot|\eta|_{-q}\,e^{\frac{3}{2}|z|_{-q}^2}e^{\frac{3}{2}|x|_{-q}^2}\|
\varphi\|_{2,p}\int_{\calS'}e^{3|y|_{-q}^2}\,\mu(dy)~\in L^1(\calS',\,\mu)\label{3.7}	
\end{align}
with respect to $x$ for any $z\in\scrC\calS_{-q}$.
Here, we remark that the assumption ``$q>\frac{\ln 6}{2\ln 2}$" guarantees the integral $\int_{\calS'}e^{3|y|_{-q}^2}\,\mu(dy)$ is finite.
Therefore,
\begin{align*}
{\rm (3.6)}=&\int_{\calS_{-q}}\left\{\sum_{n=1}^\infty\,\frac{1}{(n-1)!}\,
\int_{\calS_{-q}}D^nS\varphi(0)(x+z+\I\,y)^{n-1}\eta\,\,\mu(dy)\right\}\mu(dx)\\
=&\,S\left(\sum_{n=1}^\infty\,\frac{1}{(n-1)!}\,:D^nS\varphi(0)x^{n-1}\eta\,:\right)
(z),~~~z\in\scrC\calS_{-q},
\end{align*}
where
\begin{align}
&\left\|\sum_{n=1}^\infty\,\frac{1}{(n-1)!}\,:D^nS\varphi(0)x^{n-1}\eta\,:\right\|_{2,q}^2\nonumber\\
&~~~=\sum_{n=1}^\infty\,\frac{1}{(n-1)!}\,\|D^nS\varphi(0)(\cdot,\ldots,\cdot,\eta)
\|_{\calH\calS(\scrC\calS_{-q})}^2\nonumber\\
&~~~\leq~\omega_{p-q}^2\cdot2^{2(p-q)}\cdot\|\varphi\|_{2,p}^2\cdot |\eta|_{-p}^2.\label{3.8}
\end{align}
One notes that the inequality (\ref{3.8}) is still valid for any $p\in\r$, $q<p$ and $\eta\in\scrC\calS_{-p}$.
Since $(\calS)$ and $\scrC\calS$ are dense respectively in $(\calS_p)$ and $\scrC\calS_{-p}$ for any $p\in\r$, we can combine (\ref{3.8}) with (\ref{3.4}) and to extend the above argument to $\varphi\in (\calS)'$.
In fact, we obtain the following

\begin{proposition} \label{prop-derivative1}
Let $\varphi\in (\calS_p)$ with $p\in\r$ and $\eta\in\scrC\calS_{-p}$.
\begin{enumerate}
\item[$(i)$]
For any $q<p$, $\partial_\eta\,\varphi\in (\calS_q)$.
In fact,
$$
\partial_\eta\,\varphi=\sum_{n=1}^\infty\,\frac{1}{(n-1)!}\,:D^nS\varphi(0)x^{n-1}\eta\,:,
$$
where
$$
\|\partial_\eta\,\varphi\|_{2,q}\leq~\omega_{p-q}\cdot 2^{p-q}\cdot\|\varphi\|_{2,p}\cdot |\eta|_{-p}.
$$

\item[$(ii)$]
For $\frac{\ln 6}{2\ln 2}<q<p$, the sum in $(i)$ converges absolutely and uniformly on each bounded set in $\scrC\calS_{-q}$, and $\partial_\eta\,\varphi=\,D_\eta\,\tilde{\varphi}$.
\end{enumerate}
\end{proposition}

We commonly denote $\partial_{\delta_t}$ by $\partial_t$  for $t\in\r$, where $\delta_t$ is the Dirac measure concentrated on $t$.
If $\partial_t\,\varphi$ exists, that is, there is a $F\in (\calS)'$ such that its $U$-functional satisfies $SF=\left.\frac{d}{dw}\right|_{w=0}S\varphi(\cdot+w\,\delta_t)$, $\varphi$ is said to be $\dot{B}(t)$-differentiable and $F$ is denoted by $\partial_t\,\varphi$.
The operator $\partial_t$, which is sometimes written as ${\displaystyle \frac{\partial}{\partial\dot{B}(t)}}$, is often called the Hida derivative.

By a similar argument to (\ref{3.7}), we get the following estimation:
Assume that $\frac{\ln 6}{2\ln 2}<q<p$. Let $\varphi\in (\calS_p)$ and $\eta\in\scrC\calS_0$.
For $z\in\scrC\calS_{-q}$,
\begin{align*}
&\sum_{n=1}^\infty\,\frac{1}{(n-1)!}\,\int_{-\infty}^\infty\int_{\calS_{-p}}|\eta(t)||
(\delta_t,\,h_n)| |D^nS\varphi(0)(z+\I\,y)^{n-1}h_n|\,\mu(dy)\,dt\\
&~~\leq\,\omega_{p-q}\cdot\|\varphi\|_{2,p}\cdot|\eta|_0\cdot e^{|z|_{-q}^2}\left\{\sum_{n=0}^\infty\,(2n+2)^{-2q}\right\}^{\frac{1}{2}}
\int_{\calS'}e^{|y|_{-q}^2}\,\mu(dy).
\end{align*}
Then, for such $\varphi$, $\eta$ and $z$,
\begin{align*}
&\partial_\eta\varphi(z)\\
=&\sum_{n=1}^\infty\,\frac{1}{(n-1)!}\,\int_{\calS_{-q}}D^nS\varphi(0)(z+\I\,y)^{n-1}\eta\,\,\mu(dy)\\
=&\sum_{k=0}^\infty\,(\eta,\,h_k)\sum_{n=1}^\infty\,\frac{1}{(n-1)!}\,\int_{\calS_{-q}}D^n
S\varphi(0)(z+\I\,y)^{n-1}h_k\,\,\mu(dy)\\
=&\int_{-\infty}^\infty\eta(t)\,\sum_{n=1}^\infty\,\frac{1}{(n-1)!}\,
\int_{\calS_{-q}}\sum_{k=0}^\infty\,(\delta_t,\,h_k)\,D^nS\varphi(0)(z+\I\,y)^{n-1}h_k\,\,\mu(dy)\,dt\\
=&\int_{-\infty}^\infty\eta(t)\,\sum_{n=1}^\infty\,\frac{1}{(n-1)!}\,
\int_{\calS_{-q}}D^nS\varphi(0)(z+\I\,y)^{n-1}\delta_t\,\,\mu(dy)\,dt\\
=&\int_{-\infty}^\infty\eta(t)\,\partial_t\varphi(z)\,\,dt,
\end{align*}
where we have by Proposition \ref{prop-derivative1} that
\begin{equation}\label{3.9}
\int_{-\infty}^\infty|\eta(t)|\,\|\partial_t\varphi\|_{2,q}\,dt\leq\,\omega_{p-q}\cdot 2^{p-q}\cdot\|\varphi\|_{2,p}\cdot|\eta|_0\cdot\left\{
\sum_{n=0}^\infty(2n+2)^{-2p}\right\}^{\frac{1}{2}},
\end{equation}
which is finite provided that $p>\frac{1}{2}$.
Thus, by (\ref{3.9}) and Proposition \ref{prop-derivative1}, we can also extend the above result to $\varphi\in (\calS_p)$ with $p>\frac{1}{2}$ as follows.

\begin{proposition} \label{prop-derivative2}
\begin{enumerate}
\item[$(i)$]
For $\varphi\in (\calS_p)$ with $p>\frac{1}{2}$ and $\eta\in\scrC\calS_0$,
\begin{equation}\label{3.10}
\partial_\eta\,\varphi=\int_{-\infty}^\infty\eta(t)\,\partial_t\varphi\,\,dt ~~~\mbox{in ~$(\calS_q)$},
\end{equation}
for any $q<p$, where the right-hand integral exists in the sense of Bochner as an $(\calS_q)$-valued integral satisfying the inequality (\ref{3.9}).

\item[$(ii)$]
For $\frac{\ln 6}{2\ln 2}<q<p$, the formula (\ref{3.10}) is valid pointwise in $\scrC\calS_{-q}$:
$$
\partial_\eta\,\varphi(z)=D_\eta\,\varphi(z)=\int_{-\infty}^\infty\eta(t)\,
D_{\delta_t}\varphi(z)\,dt ~~~\mbox{$\forall$ $z\in\scrC\calS_{-q}$}.
$$
\end{enumerate}
\end{proposition}

Comparing (\ref{3.5}) with Proposition \ref{prop-derivative2} yields that $K_\varphi(t)=\partial_t\,\varphi$ for $\varphi\in (\calS_p)$ with $p>\frac{1}{2}$, where $K_\varphi$ is the kernel function given in (\ref{3.5}).
A question naturally arises:
\par
\centerline{``If $\varphi\in (L^2)$ and $\eta\in\scrC\calS_0$, how about (\ref{3.10})?"}

\vskip 0.1cm

To see it, for any $n\in\n$, let $\phi_n\in\scrC L^2(\r^n,\,dt^{\otimes n})$ be a symmetric function such that
\begin{align}
&\int\cdots\int_{\r^n}\phi_n(t_1,\ldots,t_n)\,\eta_1\hat{\otimes}\cdots\hat{\otimes}\,
\eta_n(t_1,\ldots,t_n)\,dt_1\cdots dt_n\nonumber\\
&~~~=\frac{1}{n!}\,D^nS\varphi(0)(\eta_1,\ldots,\eta_n),
~~~\eta_1,\ldots,\eta_n\in\scrC\calS_0,\label{3.11}
\end{align}
where $\hat{\otimes}$ means the symmetric tensor product.
Then, for any $n\in\n$,
$$
I_n(\phi_n)=\frac{1}{n!}\,:D^nS\varphi(0)x^n:,
$$
where $I_n(\phi_n)$ is the multiple Wiener integral of order $n$ with the kernel function $\phi_n$.
By the Fubini theorem,
$$
\sum_{n=2}^\infty\,n!\,\int\cdots\int_{\r^{n-1}}|\phi_n(t,t_2,\ldots,t_n)|^2\,dt_2\cdots dt_n<+\infty
$$
for $[dt]$-almost all $t\in\r$.
Observe that for $q<0$,
\begin{align}
&\sum_{n=1}^\infty\,n^2(n-1)!\,\int\cdots\int_{\r^{n-1}}|({\bf A}^{q})^{\otimes (n-1)}\phi_n(t,t_2,\ldots,t_n)|^2\,dt_2\cdots dt_n\nonumber\\
&~~~\leq\,(1+\omega_{-q}^2)\,\sum_{n=1}^\infty\,n!\,\int\cdots\int_{\r^{n-1}}|
\phi_n(t,t_2,\ldots,t_n)|^2\,dt_2\cdots dt_n.\label{3.12}
\end{align}
If $\varphi\in (\calS)$, it follows from (\ref{3.11}) that $\phi_n\in\scrC\calS^{\hat{\otimes} n}$ and
$$
\phi_n(t_1,\ldots,t_n)=\frac{1}{n!}\,D^nS\varphi(0)(\delta_{t_1},\ldots,\delta_{t_n})
$$
for any $t_1,\ldots,t_n\in\r$; moreover, for any $h\in\scrC\calS$,
\begin{align}
S(\partial_t\varphi)(h)=&\sum_{n=1}^\infty\,\frac{1}{(n-1)!}\,D^nS\varphi(0)h^{n-1}\delta_t\nonumber\\
=&\sum_{n=1}^\infty\,n\,\int\cdots\int_{\r^{n-1}}\phi_n(t,t_2,\ldots,t_n)\,h(t_2)\cdots h(t_n)\,dt_2\cdots dt_n\nonumber\\
=&\,S\left(\sum_{n=1}^\infty\,n\,I_{n-1}(\phi_n(t,\ldots))\right)(h),\label{3.13}
\end{align}
which implies that $\partial_t\varphi=\sum_{n=1}^\infty\,n\,I_{n-1}(\phi_n(t,\ldots))$, where it follows from (\ref{3.12}) that for any $q<0$,
\begin{equation}\label{3.14}
\left\|\sum_{n=1}^\infty\,n\,I_{n-1}(\phi_n(t,\ldots))\right\|_{2,q}^2\leq\,
(1+\omega_{-q}^2)\sum_{n=1}^\infty\,n\,\left\|I_{n-1}(\phi_n(t,\ldots))\right\|_{2,0}^2.
\end{equation}
Combine (\ref{3.13}) with Proposition \ref{prop-derivative2} and then extend $\varphi$ to $(L^2)$ by using (\ref{3.4}) and (\ref{3.14}).
Then we have
\begin{proposition} \label{prop-derivative3}
For $\varphi\in (L^2)$ and $\eta\in\scrC\calS_0$,
$$
\partial_\eta\varphi=\int_{-\infty}^\infty\eta(t)\,\left\{\,\sum_{n=1}^\infty\,n\,I_{n-1}
(\phi_n(t,\ldots))\right\}\,dt~~~\mbox{in $(\calS_q)$}
$$
for any $q<0$, where the right-hand integral exists in the sense of Bochner as an $(\calS_q)$-valued integral satisfying the inequality
$$
\int_{-\infty}^\infty|\eta(t)|\,\left\|\,\sum_{n=1}^\infty\,n\,I_{n-1}(\phi_n(t,\ldots))
\right\|_{2,q}\,dt\leq\sqrt{1+\omega_{-q}^2}\cdot|\eta|_0\cdot\|\varphi\|_{2,0}.
$$
\end{proposition}

\begin{corollary} \label{cor-kernel}
For $\varphi\in (L^2)$, the kernel function $K_\varphi(t;\cdot)$ in (\ref{3.5}) is exactly $\sum_{n=1}^\infty\,n\,I_{n-1}(\phi_n(t,\ldots))$, where, for any $q<0$,
$$
\int_{-\infty}^\infty\| K_\varphi(t;\cdot)\|_{2,q}^2\,dt\leq\,(1+\omega_{-q}^2)\cdot\|\varphi\|_{2,0}^2.
$$
\end{corollary}

\section{Integration by Parts Formula}

For any probability measure $\calP$ on $(\calS',\,\scrB(\calS'))$, let $\calG_\calP$ be the class consisting of all complex-valued functions $\varphi$ on $\calS'$ satisfying the conditions: For any $h\in\calS_0$, (a) $\varphi$ is G\^ateaux differentiable at $x$ in the direction of $h$ for any $x\in\calS'$; (b) both $\varphi$ and $\delta\varphi(\cdot;\,h)$ belong to $\scrC L^\alpha(\calS',\,\calP)$ for some $\alpha>1$, where  $\delta\varphi(x;\,h)$ the G\^ateaux derivative of $\varphi$ at $x\in\calS'$ in the direction of $h$.

For any $h\in\calS_0$, let $e(h)=\,e^{\l\cdot,\,h\g-\frac{1}{2}\int_{-\infty}^\infty |h(t)|^2\,dt}$.  Then, for $\varphi\in\calG_\mu$,
\begin{align*}
\int_{\calS'}\frac{\varphi(x+rh)-\varphi(x)}{r}\,\mu(dx)=&\int_{\calS'}\varphi(x)
\cdot\left\{\frac{e(rh)(x)-1}{r}\right\}\,\mu(dx)\\
\to& \int_{\calS'}\l x,\,h\g\,\varphi(x)\,\mu(dx),~~~\mbox{as $r\to 0^+$},
\end{align*}
where the last term is obtained by the inequality that
\begin{equation*}
\left|\frac{e(rh)(x)-1}{r}\right|\leq \left\{|\l x,\,h\g|+\int_{-\infty}^\infty| h(t)|^2\,dt\right\}\cdot e^{|\l x,\,h\g|}\in L^{\alpha'}\!(\calS',\,\mu)
\end{equation*}
for any $0<z<1$, $\frac{1}{\alpha}+\frac{1}{\alpha'}=1$, and then applying the dominated convergence argument. On the other hand,
\begin{align*}
\int_{\calS'}\frac{\varphi(x+rh)-\varphi(x)}{r}\,\mu(dx)=&\int_{\calS'}\frac{1}{r}\int_0^r\delta \varphi(x+\vartheta\,h;\,\eta)\,d\vartheta\,\mu(dx)\\
=&\int_{\calS'}\delta\varphi(x;\,h)\cdot\left\{\frac{1}{r}\int_0^r\,e(\vartheta\,h)\,
d\vartheta\right\}\,\mu(dx)\\
\to&\int_{\calS'}\delta\varphi(x;\,h)\,\mu(dx),~~~\mbox{as $r\to 0^+$,}
\end{align*}
where the last term is obtained by applying the dominated convergence argument.
Putting together the above formulas shows the following integration by parts formula for the white noise measure $\mu$.

\begin{theorem} \label{thm-int-by-parts}
Let $\varphi:\calS'\to\c$ be in the class $\calG_\mu$.
Then, for any $h\in\calS_0$,
\begin{equation}\label{4.1}
\int_{\calS'} \l x,\, h\g\,\varphi(x)\,\mu(dx)=\int_{\calS'} \delta\varphi(x;\,h)\,\mu(dx).
\end{equation}
\end{theorem}

\begin{remark}
The integration by parts formula for abstract Wiener measures was obtained by Kuo \cite{Ka} in 1974.
Recently, under much weak conditions, Kuo and Lee \cite{LK} reformulated this formula and simplified the proof by applying the technique Stein used in proving his famous Stein's lemma (Proposition \ref{prop-characterization}) for normal distribution (see \cite{{Steina},{Steinb}}).
\end{remark}

The formula (\ref{4.1}) is an infinite dimensional analogue of the Stein identity for normal distribution, which also characterizes the white noise measure as follows.

\begin{theorem} \label{char-wn}
A probability measure $\calP$ on $(\calS',\,\scrB(\calS'))$ is equal to the white noise measure $\mu$ if and only if for any $\varphi\in\calG_\calP$, the following equality holds:
\begin{equation}\label{4.2}
\int_{\calS'}\bigl\{(x,\,\eta)\,\varphi(x)-\delta\varphi(x;\,\eta)\bigr\}\,\calP(dx)=0,~~~\forall~\eta\in\calS.
\end{equation}
\end{theorem}

\begin{proof}
If $\calP=\mu$, then relation (\ref{4.2}) is satisfied by virtue of Theorem 5.1.
Now, suppose $\calP$ satisfies (\ref{4.2}).
Let
$$
\phi_\eta(r)=\int_{\calS'}e^{\I r(x,\,\eta)}\,\calP(dx),~~~r\in\r,~\eta\in\calS.
$$
By the mean value theorem for differentiation, there are two real numbers $p_{sr},q_{sr}$ between $s$ and $r$ such that
$$
e^{\I \,s(x,\,\eta)}-e^{\I\, r(x,\,\eta)}=\frac{\I\,(s-r)}{2}\cdot(x,\eta)\cdot \Psi_{s,r;\eta}(x),~~~x\in\calS',
$$
where $\Psi_{s,r;\eta}(x)=\,e^{\I \,p_{sr}(x,\,\eta)}-e^{-\I\, p_{sr}(x,\,\eta)}+e^{\I\, q_{sr}(x,\,\eta)}+e^{-\I\, q_{sr}(x,\,\eta)}$. Then
\begin{equation*}
\frac{\phi_\eta(s)-\phi_\eta(r)}{s-r}=(\I/2)\int_{\calS'}\,(x,\eta)\cdot\Psi_{s,r;\eta}(x)\,\calP(dx).
\end{equation*}
It is clear that $\Psi_{s,r;\eta}\in\calG_\calP$ Then, by (\ref{4.2}) we have
\begin{align*}
\int_{\calS'}\,(x,\eta)\cdot\Psi_{s,r;\eta}(x)\,\calP(dx)=\,\I\,\sum_{j=1}^4a_jv_j\,\phi_\eta(v_j)\,|\eta|_0^2,
\end{align*}
where $v_1=p_{sr}=-v_2$, $v_3=q_{sr}=-v_4$, $a_1=a_3=a_4=1,a_2=-1$. Letting $s$ tend to $r$, we see that $\frac{d \phi_\eta(r)}{dr}$ exists, and $\frac{d \phi_\eta(r)}{dr}=\!-r\phi_\eta(r)\,|\eta|_0^2$ with $\phi_\eta(0)=1$.
Thus, $\phi_\eta(r)=\,e^{-\frac{1}{2}r^2\,|\eta|_0^2}$ and $\calP=\mu$.
\end{proof}

 By Theorem \ref{char-wn} and an observation of its proof, we can recover the following version of Stein's characterization of the normal distribution.

\begin{corollary} {\rm $($Stein's lemma$)$} A real-valued random variable $Y$ has the standard normal distribution if and only if
\begin{equation}\label{4.3}
\e[\,Yf(Y)]=\e[f'(Y)]
\end{equation}
for any bounded complex-valued function $f$ with bounded derivative $f'$.
\end{corollary}

Let $f$ be a function defined on $\calS'$ with values in a complex Banach space $W$. Then $f$ is said to be $\calS_0$-differentiable if the mapping $\phi(h)\equiv f(x+h)$, $h\in \calS_0$, is Fr\'echet differentiable at $0$ for any $x\in\calS'$. The Fr\'echet derivative $\phi'(0)$ at $0\in \calS_0$ is called the $\calS_0$-derivative of $f$ at $x\in \calS'$, denoted by $\l Df(x),\,h\g$. The $k$-th order $\calS_0$-derivatives of $f$ at $x$ are defined inductively and denoted by $D^kf(x)$ for $k\geq 2$ if they exist. One notes that $D^kf(x)$ is a bounded $k$-linear mapping from the Cartesian product $\calS_0\times\cdots\times \calS_0$ of $k$ copies of $\calS_0$ into $W$ for any $k\in\n$. In particular, when $W=\r$, $Df(x)\in \calS_0$ and $D^2f(x)$ is regarded as a bounded linear operator from $\calS_0$ into $\calS_0$ for any $x\in\calS'$ (see \cite{Kb}).

We can use Theorem \ref{thm-int-by-parts} to obtain another integration by parts formula which also characterizes the white noise measures.

\begin{theorem} \label{thm-int-by-parts2} {\rm $($cf. \cite{LK}$)$}
 A probability measure $\calP$ on $(\calS',\,\scrB(\calS'))$ is equal to the white noise measure $\mu$  if and only if, for any $\calS_0$-differentiable function $f$ on $\calS'$ such that $f(x)\in\scrC\calS_p$, $p>\frac{1}{2}$, for any $x\in\calS'$, $\|D f(\cdot)\|_{\rm tr} \in
 L^1(\calS',\,\calP)$ and $\int_{\calS'} |f(x)|^{\alpha}_p\,\calP(dx)<\infty$ for some $\alpha>1$,
 the following equality holds:
\begin{equation} \label{4.4}
\int_{\calS_{-p}} (x,\,f(x))_p\,\calP(dx) =\int_{\calS'}
 {\rm Tr} (D f(x))\,\calP(dx),
\end{equation}
where $(\cdot, \cdot)_p$ is the $\scrC\calS_{-p}$-$\scrC\calS_p$ pairing,
${\rm Tr}(\cdot)$ denotes the trace of a trace class operator on $\calS_0$ and $\|\cdot\|_{\rm tr}$ means
 the trace class norm.
\end{theorem}

\begin{proof}
 {\it Necessity.} Let $\{h_n\}_{n=0}^\infty$ be the CONS for $\calS_0$ as mentioned in Subsection 3.1, and for any $x\in\calS'$ and $n\in\n_0$, let $P_nx=\sum_{j=0}^n\,(x,\,h_j)\,h_j$.
Since $f$ is $\calS_0$-differentiable, $r^{-1}(f(x+rh)-f(x))$ converges to $\l Df(x),\,h\g$ uniformly with respect to $h$ on each bounded set in $\calS_0$ as $r\to 0$ for any $x\in\calS'$. This implies that $\varphi(x)\equiv (f(x),\,h_j)$, $x\in\calS'$, is G\^ateaux differentiable at $x$ in the direction of $h\in\calS_0$, and $\delta\varphi(x;\,h)=\bigl\l\l Df(x),\,h\g,\,h_j\bigr\g_0$ for any $j\in\n_0$. Moreover, it follows from the conditions (a) and (b) and by applying the Fernique theorem (see \cite{Kb}) that $\varphi$ is in $\calG_\mu$. Then, by applying Theorem 5.1, we see that
\begin{align} \label{4.5}
\int_{\calS_{-p}} (P_nx,\,f(x))_p \,\mu(dx) =&\sum_{j=0}^n\int_{\calS_{-p}}
  ( f(x),\,h_j )(x,\,h_j) \,\mu(dx)\nonumber\\
=& \sum_{j=0}^n\int_{\calS_{-p}}
  \bigl\l\l Df(x),\,h_j\g,\,h_j\bigr\g_0 \,\mu(dx) \nonumber  \\
  =& \int_{\calS_{-p}} {\rm Tr}\big(\l Df(x),\,P_n(\cdot)\g\bigr)\,\mu(dx).
\end{align}
Note that for all $x\in \calS_{-p}$ and $n\in\n_0$,
 \[
 \big|{\rm Tr}\big(\l Df(x),\,P_n(\cdot)\g\bigr)\big|
 \; \le \; \|Df(x)\|_{\rm tr},
 \]
 \[
 \big|\left(P_nx,\,f(x)\right)\big|
 \; \le \; |
 f(x)|_p\,|P_n x|_{-p}
 \; \le \; |f(x)|_p\,|x|_{-p}.
 \]
Let $n$ tend to infinity, and then obtain Equation (\ref{4.4}) by applying the Lebesgue dominated convergence
theorem to (\ref{4.5}).

{\it Sufficiency.} Fix $\eta\in\calS$. Let $f(x)=\I\,e^{\I\,(x,\,\eta)}\eta$ for any $x\in\calS'$. Then $f$ is Fr\'echet differentiable on $\calS'$, and $\l Df(x),\,y\g=-(y,\,\eta)\,e^{\I\,(x,\,\eta)}\eta$ for any $y\in\calS'$. Thus $Df(x)$ is a bounded linear operator from $\calS_{-p}$ into $\scrC\calS_p$ for $p>\frac{1}{2}$, the operator norm of which is less than or equal to $|\eta|_p^2$. By applying the Goodman theorem and Fernique theorem (see \cite{Kb}),
$$
\int_{\calS'}\| Df(x)\|_{{\rm tr}}\,\calP(dx)\leq |\eta|_p^2\int_{\calS_{-p}}|x|_{-p}^2\,\mu(dx)<+\infty.
$$
 In addition, for any $1\leq \alpha<+\infty$, $\int_{\calS'} |f(x)|^{\alpha}_p\,\calP(dx)=|\eta|_p^\alpha<+\infty$. By the assumption, $f$ satisfies the identity (\ref{4.4}) and we have
\begin{equation*}
  \int_{\calS'}( x,\,\eta)\,e^{\I\,(x,\,\eta)}\,\calP(dx)=\I\,|\eta|_0^2\,\int_{\calS'}e^{\I\,(x,\,\eta)}\,\calP(dx).
\end{equation*}
 By the same argument as in the proof of Theorem \ref{char-wn}, $\int_{\calS'}e^{\I\,(x,\,\eta)}\,\calP(dx)=e^{-\frac{1}{2}|\eta|_0^2}$. The proof is complete.
\end{proof}

\begin{remark}
In Theorem \ref{thm-int-by-parts2}, the assumption ``$p>\frac{1}{2}$" is necessary based on the fact that $(\calS_0,\,\calS_{-p})$, $p>\frac{1}{2}$, is an abstract Wiener space.
\end{remark}

\vskip 0.2cm

\noindent
{\it $-$~Application to Number operators}

\vskip 0.2cm

Let $f$ be a complex-valued function on $\calS'$. If $f$ is twice $\calS_0$-differentiable at $x\in\calS'$ and $D^2f(x)$ is a trace-class operator on $\calS_0$, its trace is known as the Gross Laplacian $\Delta_{{}_G}f(x)$ of $f$ at $x$: $\Delta_{{}_G}f(x)={\rm Tr}(D^2f(x))$ (see \cite{Gross}).
In particular, if $f$ is twice Fr\'echet differentiable in $\calS_{-p}$ with $p>\frac{1}{2}$, then the restriction $D^2f(x)|_{\calS_0}$ of $D^2f(x)$ to $\calS_0$ is automatically of trace class on $\calS_0$ by the Goodman theorem (see \cite{Kb}).

Now, if $f$ is twice $\calS_0$-differentiable at $x\in\calS_{-p}$, $p>\frac{1}{2}$, such that $Df(x)\in\calS_p$ and $D^2f(x)$ is a trace-class operator on $\calS_0$, we define the Beltrami Laplacian
$$
\Delta_{{}_B} f(x)=\Delta_{{}_G}f(x)-(x,\,Df(x))_p.
$$

For $\varphi\in (\calS)$ and $x\in\calS'$, $D\varphi(x)$ is a continuous linear functional on $\calS'$.
Since $\calS$ is a nuclear space, $D\varphi(x)\in\calS$.
Similarly, $D^2\varphi(x)$ is a continuous linear operator from $\calS'$ into $\calS$.
Then it follows from the Goodman theorem that $D^2\varphi(x)|_{\calS_0}$ is a trace-class operator on $\calS_0$, and thus $\Delta_{{}_B}\varphi(x)$ exists.

As a consequence of Theorem \ref{thm-int-by-parts2}, we can see that for any $n\in\n_0$,
$$
\Delta_{{}_B}\int_{\calS'}\prod_{j=1}^n(x+\I\,y,\,h_{k_j})\,\mu(dy)=-n\int_{\calS'}\prod_{j=1}^n(x+\I\,y,\,h_{k_j})\,\mu(dy),
$$
where $k_j$'s $\in\n_0$.
See also \cite{La}.
In fact, $\Delta_{{}_B}$ is densely defined on $(L^2)$.
The closure of the operator $-\Delta_{{}_B}$, denoted by $\calN$, is known as the number operator. Then the domain Dom($\calN$) of $\calN$ is
$$
{\rm Dom}(\,\calN)=\left\{\varphi\in (L^2);\,\,\sum_{n=0}^\infty\,\frac{n^2}{n!}\|D^nS\varphi(0)\|_{\calH\calS^n(\scrC\calS_0)}^2<+
\infty\right\}.
$$
It is obvious that $(\calS_p)\subset {\rm Dom}(\calN)$ for $p\geq \frac{1}{2}$.

Now, let $\varphi,\psi\in (\calS)$. By Theorem \ref{thm-int-by-parts2},
\begin{align}
&\int_{\calS'}\varphi(x)\,\calN\psi(x)\,\mu(dx)\nonumber\\
&~~~=-\int_{\calS'}\varphi(x)\,\Delta_{{}_G}\psi(x)\,\mu(dx)+\int_{\calS'}\varphi(x)\,(x,\,D\psi(x))\,\mu(dx)\nonumber\\
&~~~=-\int_{\calS'}\varphi(x)\,\Delta_{{}_G}\psi(x)\,\mu(dx)+\int_{\calS'}(x,\,\varphi(x)D\psi(x))\,\mu(dx)\nonumber\\
&~~~=-\int_{\calS'}\varphi(x)\,\Delta_{{}_G}\psi(x)\,\mu(dx)+\int_{\calS'}{\rm Tr}(D(\varphi(\cdot)D\psi(\cdot))(x))\,\mu(dx).\label{4.6}
\end{align}
 Observe that for any $y,z\in\calS'$,
\begin{align*}
( D_y(\varphi(\cdot)D\psi(\cdot))(x),\,z)=\,& D_y\varphi(x)\, D_z\psi(x)+\varphi(x)\,D^2\psi(x)(y,z)\\
                             =\,& \partial_y\varphi(x)\, \partial_z\psi(x)+\varphi(x)\,D^2\psi(x)(y,z).
\end{align*}
Then, for any $x\in\calS'$, it follows from Proposition \ref{prop-derivative2} that
\begin{align}\label{4.7}
&{\rm Tr}(D(\varphi(\cdot)D\psi(\cdot))(x))\nonumber\\
&~~~=\sum_{n=0}^\infty\partial_{h_n}\varphi(x)\, \partial_{h_n}\psi(x)+\varphi(x)\,\Delta_{{}_G}\psi(x)\nonumber\\
&~~~=\sum_{n=0}^\infty\int_{-\infty}^\infty h_n(t)\,\partial_t\varphi(x)\,dt\int_{-\infty}^\infty h_n(t)\, \partial_t\psi(x)\,dt+\varphi(x)\,\Delta_{{}_G}\psi(x)\nonumber\\
&~~~=\int_{-\infty}^\infty\partial_t\varphi(x)\, \partial_t\psi(x)\,dt+\varphi(x)\,\Delta_{{}_G}\psi(x).
\end{align}
Combining (\ref{4.6}) with (\ref{4.7}), applying Proposition \ref{prop-derivative1} and by extension, we have the following

\begin{theorem} \label{thm-inner-prod}
For any $\varphi,\psi\in (\calS_p)$ with $p>\frac{1}{2}$,
$$
\ll\,\calN\varphi,\,\psi\gg_{2,0}=\int_{-\infty}^\infty\ll\partial_t\varphi,\,\partial_t\psi\gg_{2,0}\,dt,
$$
where $\ll\cdot,\cdot\gg_{2,0}$ is the inner product induced by $\|\cdot\|_{2,0}$-norm. Moreover,
$$
\int_{-\infty}^\infty\|\partial_t\,\varphi\|_{2,0}\|\partial_t\,\psi\|_{2,0}\,dt\leq\,\omega_p^2\cdot 2^{2p}\cdot\|\varphi\|_{2,p}\cdot\|\psi\|_{2,p}\left\{\sum_{n=0}^\infty\,(2n+2)^{-2p}\right\}.
$$
\end{theorem}

Observe that for $\varphi\in (L^2)$, it is analogous to Corollary 3.5 that we have
$$
\int_{-\infty}^\infty \| K_\varphi(t;\cdot)\|_{2,0}^2\,dt=\sum_{n=1}^\infty\,n\,\|I_n(\phi_n)\|_{2,0}^2.
$$
Then we have more general results than those in Proposition \ref{prop-derivative3}, Corollary \ref{cor-kernel} and Theorem \ref{thm-inner-prod} as follows.

\begin{theorem} \label{thm-derivative}
\begin{enumerate}
  \item[$(i)$] Let $\varphi\in {\rm Dom}(\,\calN^{1/2})$. Then $K_\varphi(t)\in (L^2)$ for $[dt]$-almost all $t\in\r$ and satisfies
  $$
\int_{-\infty}^\infty\|K_\varphi(t;\cdot)\|_{2,0}^2\,dt=\,\|\calN^{1/2}\varphi\|_{2,0}^2.
$$
Moreover, for any $\eta\in\scrC\calS_0$,
$$
\partial_\eta\varphi=\int_{-\infty}^\infty\eta(t)\,K_\varphi(t;\,\cdot)\,dt~~~\mbox{in $(L^2)$},
$$
where
$$
\|\partial_\eta\varphi\|_{2,0}\leq\,|\eta|_0\cdot\|\calN^{1/2}\varphi\|_{2,0}.
$$
  \item[$(ii)$] Let $\varphi\in{\rm Dom}(\,\calN)$ and $\psi\in (L^2)$. Then the sum
  $$
 \sum_{n=1}^\infty n^2\ll I_{n-1}(\phi_n(t,\ldots)),\,I_{n-1}(\psi_n(t,\ldots))\gg_{2,0},
$$
denoted by $[ K_\varphi(t;\cdot),\,K_\psi(t;\cdot)]_{2,0}$, absolutely converges for $[dt]$-almost all $t\in\r$, where $\psi_n$ is the kernel function defined analogously as $\phi_n$ in (\ref{3.11}) by replacing $\varphi$ by $\psi$. Moreover,
$$
\ll\,\calN\varphi,\,\psi\gg_{2,0}=\int_{-\infty}^\infty [K_\varphi(t;\cdot),\,K_\psi(t;\cdot) ]_{2,0}\,dt.
$$
 \end{enumerate}
\end{theorem}

\vskip 0.2cm

\noindent{\bf Note.}~In the sequel, we will identify $K_\varphi(t;\cdot)$ with $\partial_t\varphi$ for $\varphi\in (L^2)$.

\begin{remark}
If both $K_\varphi(t;\cdot)$ and $K_\psi(t;\cdot)$ are in $(L^2)$, then
$$
[ K_\varphi(t;\cdot),\,K_\psi(t;\cdot)]_{2,0}=\ll  K_\varphi(t;\cdot),\,K_\psi(t;\cdot)\gg_{2,0}.
$$
In fact, ${\rm Dom}(\calN^{1/2})={\mathbb D}^{1,2}$, and for $\varphi\in{\rm Dom}(\calN^{1/2})$, $K_\varphi(\cdot,x)$ coincides with the Malliavin derivative $D\varphi(x)$ of $\varphi$, $x\in\calS'$.
\end{remark}

\section{Connecting Stein's Method with Hida Calculus for White Noise Functionals}

 Let $\scrH$ be a separating class of Borel-measurable complex-valued test functions on $\r$, which means that  any two real-valued random variables $F,G$ satisfying $\e[h(F)]=\e[h(G)]$ for every $h\in\scrH$ have the same law. Of course $F$ and $G$ are assumed to be such that both $\e h(F)$ and $\e h(G)$ exists for all $h \in \scrH$. For any two such real-valued random variables $F$ and $G$, the distance between the laws of $F$ and $G$, induced by $\scrH$, is given by
$$
d_\scrH(F,\,G)=\sup\{|\e[h(F)]-\e[h(G)]|;\,h\in\scrH\}.
$$

Let $\scrH$ be such that $\e h(Z)$ exists for $h \in \scrH$, where $Z$ has the standard normal distribution. Let $f_h$ be the solution, given by (\ref{Stein-eqn-soln}), of the Stein equation
\begin{equation}
f'(w) - wf(w) = h(w) - \e h(Z).
\end{equation}
Assume that $\scrH$ is such that $f_h$ is bounded and absolutely continuous with bounded $f'_h$ for $h \in \scrH$. Then for any random variable $Y$ such that $\e h(Y)$ exists for $h \in \scrH$, we have
$$
d_\scrH(Y,\,Z)=\,\sup\{|\e[f_h'(Y)-Yf_h(Y)]|;\,h\in\scrH\}.
$$

Consider $(\calS',\scrB(\calS'),\mu)$ as the underlying probability space. Assume that $\varphi\in{\rm Dom}(\,\calN^{1/2})$ with $\e[\varphi]=0$. Define
$$
\calN^{-1}\varphi=\sum_{n=1}^\infty\,\frac{1}{nn!}\,:D^nS\varphi(0)x^n:.
$$
Then $\varphi=\calN\calN^{-1}\varphi$ since $\e[\varphi]=0$.
It is noted that $\calN^{-1}\varphi\in {\rm Dom}(\calN)$ and $f_h(\varphi)\in (L^2)$. By Theorem \ref{thm-derivative},
\begin{align}
\int_{\calS'}\varphi(x)\,f_h(\varphi(x))\,\mu(dx)=&\,\ll\,\calN\calN^{-1}\varphi,\,
f_h(\varphi)\gg_{2,0}\nonumber\\
=&\int_{-\infty}^\infty [\,\partial_t\,\calN^{-1}\varphi,\,\partial_t f_h(\varphi)]_{2,0}\,\,dt. \label{5.1}
\end{align}

Take a sequence $\{\varphi_k\}\subset (\calS)$ such that $\calN^{1/2}\varphi_k\to \calN^{1/2}\varphi$ in $(L^2)$.
Then
$$
\| f_h(\varphi_k)-f_h(\varphi)\|_{2,0}\leq\,|f_h'|_\infty\cdot\| \varphi_k-\varphi\|_{2,0}\to 0,~~~\mbox{as $k\to\infty$},
$$
and thus
\begin{equation}\label{5.2}
\ll\,\calN\calN^{-1}\varphi,\,f_h(\varphi_k)\gg_{2,0}\to\ll\,\calN\calN^{-1}\varphi,\,
f_h(\varphi)\gg_{2,0}~~~\mbox{ as $k\to\infty$}.
\end{equation}
On the other hand, for each $k\in\n$, it follows by the chain rule that
\begin{align*}
\left.\frac{d}{dw}\right|_{w=0}Sf_h(\varphi_k)(z+w\delta_t)=&\left.\frac{d}{dw}\right|_{w=0}\int_{\calS'}f_h(\varphi_k)(z+w\delta_t+x)\,\mu(dx)\\
=&\int_{\calS'}f_h'(\varphi_k(z+x))\,\partial_t\varphi_k(z+x)\,\mu(dx)\\
=&~S(f_h'(\varphi_k)\,\partial_t\varphi_k)(z),~~~z\in\scrC\calS_0,
\end{align*}
and then
\begin{equation}\label{5.3}
\int_{-\infty}^\infty [\,\partial_t\,\calN^{-1}\varphi,\,\partial_t f_h(\varphi_k)]_{2,0}\,\,dt=\int_{-\infty}^\infty \ll\,\partial_t\,\calN^{-1}\varphi,\,f_h'(\varphi_k)\,\partial_t\varphi_k\gg_{2,0}\,\,dt.
\end{equation}
 By Theorem \ref{thm-derivative}, $f_h'(\varphi)\,\partial_t\varphi\in (L^2)$ for almost all $t\in\r$. We would like to estimate
\begin{align}
&\left|\int_{-\infty}^\infty \ll\,\partial_t\,\calN^{-1}\varphi,\,f_h'(\varphi_k)\,\partial_t\varphi_k-f_h'(\varphi)\,\partial_t\varphi\gg_{2,0}\,\,dt\right|\nonumber\\
&~~~~~~\leq\int_{-\infty}^\infty |\ll\,\partial_t\,\calN^{-1}\varphi,\,(f_h'(\varphi_k)-f_h'(\varphi))\,\partial_t\varphi_k\gg_{2,0}|\,\,dt\nonumber\\
&~~~~~~~~~+\int_{-\infty}^\infty |\ll\,\partial_t\,\calN^{-1}\varphi,\,f_h'(\varphi)(\,\partial_t\varphi_k-\partial_t\,\varphi)\gg_{2,0}|\,\,dt\nonumber\\
&~~~~~~\leq\,3|f_h'|_\infty\int_{-\infty}^\infty \ll\,|\partial_t\,\calN^{-1}\varphi|,\,|\partial_t\varphi_k-\partial_t\varphi|\gg_{2,0}\,\,dt\label{5.4}\\
&~~~~~~~~~+\int_{-\infty}^\infty \ll\,|\partial_t\,\calN^{-1}\varphi|,\,|f_h'(\varphi_k)-f_h'(\varphi)|\,|\partial_t\varphi|\gg_{2,0}\,\,dt\label{5.5}.
\end{align}

By Theorem \ref{thm-derivative},
\begin{align*}
\int_{-\infty}^\infty\|\partial_t\varphi_k-\partial_t\varphi\|_{2,0}^2\,dt=\|\calN^{1/2}( \varphi_k-\varphi)\|_{2,0}^2\to 0.
\end{align*}
Then
$$
(\ref{5.4})\leq\,3|f_h'|_\infty \left\{\int_{-\infty}^\infty \|\partial_t\,\calN^{-1}\varphi\|_{2,0}^2\,dt\right\}^{\frac{1}{2}}\left\{\int_{-\infty}^\infty\|\partial_t\varphi_k-\partial_t\varphi\|_{2,0}^2\,\,dt\right\}^{\frac{1}{2}}\to 0
$$
 as $k\to\infty$. In addition, since $\|\varphi_k\to\varphi\|_{2,0}\to 0$, there exists a subsequence, still written as $\{\varphi_k\}$, such that
$$
\varphi_k(x)\to\varphi(x),~~~\mbox{for $[\mu]$-almost all $x\in\calS'$.}
$$
Hence
$$
\Phi_k(t,x):=|\partial_t\,\calN^{-1}\varphi(x)|\,|f_h'(\varphi_k(x))-f_h'(\varphi(x))|\,|\partial_t\varphi(x)|\to 0
$$
for $[dt\otimes\mu]$-almost all $(t,x)$ in $\r\times\calS'$, where
$$
|\Phi_k(t,x)|\leq\,2|f_h'|_\infty|\partial_t\,\calN^{-1}\varphi(x)|\,|\partial_t\varphi(x)|\in L^1(\r\times\calS',\,dt\otimes\mu).
$$
Then it follows by the Lebesgue dominated convergence theorem that (\ref{5.5})$\to 0$ as $k\to\infty$. So we can conclude by (\ref{5.2})--(\ref{5.5}) that
\begin{align*}
\ll\,\calN\calN^{-1}\varphi,\,f_h(\varphi)\gg_{2,0}=&\lim_{k\to\infty}\ll\,\calN\calN^{-1}\varphi,\,f_h(\varphi_k)\gg_{2,0}\\
=&\lim_{k\to\infty}\int_{-\infty}^\infty \ll\,\partial_t\,\calN^{-1}\varphi,\,\partial_t f_h(\varphi_k)\gg_{2,0}\,\,dt\\
=&\lim_{k\to\infty}\int_{-\infty}^\infty \ll\,\partial_t\,\calN^{-1}\varphi,\, f_h'(\varphi_k)\,\partial_t\varphi_k\gg_{2,0}\,\,dt\\
=&\int_{-\infty}^\infty \ll\,\partial_t\,\calN^{-1}\varphi,\, f_h'(\varphi)\,\partial_t\varphi\gg_{2,0}\,\,dt\\
=&\int_{\calS'}f_h'(\varphi(x))\int_{-\infty}^\infty \partial_t\,\calN^{-1}\varphi(x)\, \partial_t\varphi(x)\,\,dt\,\mu(dx).
\end{align*}
Together with (\ref{5.1}) and $d_\scrH(\varphi,\,Z)$, we obtain

\begin{theorem} \label{thm-normal-approx-bd}
Let $\varphi\in {\rm Dom}(\,\calN^{1/2})$ with $\e[\varphi]=0$ and let $Z$ on $(\calS',\scrB(\calS'),\mu)$ have the standard normal distribution.  Suppose $\scrH$ is a separating class of compex-valued test functions defined on $\r$ such that for $h \in \scrH$, both $\e h(\varphi)$ and $\e h(Z)$ exist and $f_h$ is a bounded function with bounded continuous derivative $f'_h$. Then we have
$$
d_{\scrH}(\varphi,\,Z)\leq\,\left\{\sup_{h\in\scrH}|f_h'|_\infty\right\}\int_{\calS'}\left|1-\int_{-\infty}^\infty\partial_t\,\calN^{-1}\varphi(x)\cdot\partial_t\varphi(x)\,dt\right|\,\mu(dx).
$$
\end{theorem}

If $\varphi(x)=(x,\,\eta)$ for $\eta\in\calS$ and $x\in\calS'$, then $\varphi$ has the normal distribution with mean $0$ and variance $|\eta|_0^2$. It is easy to see that
\begin{align*}
&\,\left\{\sup_{h\in\scrH}|f_h'|_\infty\right\}\int_{\calS'}\left|1-\int_{-\infty}^\infty\partial_t\,\calN^{-1}\varphi(x)\cdot\partial_t\varphi(x)\,dt\right|\,\mu(dx)\\
&~~~=\,\left\{\sup_{h\in\scrH}|f_h'|_\infty\right\}\cdot\bigl|1-|\eta|_0^2\bigr|.
\end{align*}

So, if $\varphi\sim N(0,1)$, the upper bound in Theorem \ref{thm-normal-approx-bd} is zero.
This shows that the bound is tight.
The quantity $\sup_{h\in\scrH}|f_h'|_\infty$ can be bounded by applying Proposition \ref{prop-Stein-eqn-soln}.
The function $f'_h$ is continuous if $h$ is.
For total variation distance, (\ref{d_TV-h-C}) allows $h$ to be such that both the real and imaginary parts of it are continuous and bounded between $0$ and $1$.

Theorem \ref{thm-normal-approx-bd} provides a general bound for the normal approximation for $\varphi\in {\rm Dom}(\,\calN^{1/2})$.
It can be applied to produce explicit bounds for special cases of $\varphi$ as in the work of Nourdin and Peccati \cite{NP}.
However, we will not pursue this application in this chapter.

\section*{Acknowledgement}

This work was partially supported by Grant R-146-000-182-112 from the National University of Singapore and also by  Grant 104-2115-M-390-002 from the Ministry of Science and Technology of Taiwan.

\clearpage
\thispagestyle{empty}
\markboth{}{}
\cleardoublepage

\end{document}